\documentclass[11pt,reqno]{amsart}
\usepackage{amsmath,amssymb,latexsym,esint,cite,mathrsfs}
\usepackage{verbatim,wasysym}
\usepackage[left=1.8cm,right=1.8cm,top=2.4cm,bottom=2.4cm]{geometry}
\usepackage{tikz,enumitem,graphicx, subfig, microtype, color}
\usepackage{epic,eepic}

\usepackage[colorlinks=true,urlcolor=blue, citecolor=red,linkcolor=blue,
linktocpage,pdfpagelabels, bookmarksnumbered,bookmarksopen]{hyperref}
\usepackage[hyperpageref]{backref}
\usepackage[english]{babel}
\usepackage{appendix}

\numberwithin{equation}{section}

\newtheorem{thm}{Theorem}[section]
\newtheorem{lem}[thm]{Lemma}
\newtheorem{cor}[thm]{Corollary}
\newtheorem{Prop}[thm]{Proposition}

\newtheorem{Rem}[thm]{Remark}

\newtheorem{step}[thm]{Step}

\begin{document}
	\title[The nonlinear elliptic equations with critical Sobolev exponent]
	{Qualitative properties of blowing-up solutions
of nonlinear elliptic equations with critical Sobolev exponent}

	\author[M. Yang]{Minbo Yang$^\dag$}
	\address{\noindent Minbo Yang  \newline
		School of Mathematical Sciences, Zhejiang Normal University,\newline
		Jinhua 321004, Zhejiang, People's Republic of China.}\email{mbyang@zjnu.edu.cn}
	
	\author[S. Zhao]{Shunneng Zhao$^\ddag$}
	\address{\noindent Shunneng Zhao  \newline
School of Mathematical Sciences, Zhejiang Normal University,\newline
		Jinhua 321004, Zhejiang, People's Republic of China.
		\vspace{2mm}
		\newline
  Dipartimento di Matematica, Universit\`{a} degli Studi di Bari Aldo Moro,\newline Via Orabona 4, 70125 Bari, Italy.  }
	\email{snzhao@zjnu.edu.cn}

	\thanks{2020 {\em{Mathematics Subject Classification.}} Primary 35B40; 35B33;  Secondly  35J15; 35J60.}
	
	\thanks{{\em{Key words and phrases.}} Linearized problem; Critical exponents; Asymptotic behavior of solutions; Nondegeneracy.}
	
	\thanks{$^\dag$Minbo Yang was partially supported by the National Key Research and Development Program of China (No. 2022YFA1005700) and National Natural Science Foundation of China (12471114).}
	\thanks{$^\ddag$Shunneng Zhao was partially supported by National Natural Science Foundation of China (12401146, 12261107) and Natural Science Foundation of Zhejiang Province (LMS25A010007) and PNRR MUR project PE0000023 NQSTI - National Quantum Science and Technology Institute (CUP H93C22000670006).}
	
	\allowdisplaybreaks
	
	\begin{abstract}
		{\small
			In this paper, we are concerned with the critical elliptic equation
\begin{equation}\label{kx}
\left\lbrace
\begin{aligned}
    &-\Delta u=u^{p}+\epsilon \kappa(x)u^{q}\quad\hspace{2mm} \mbox{in}~~\Omega,
    \\&u>0\quad \quad\quad\quad\quad\quad\quad\quad\hspace{1mm}\hspace{0.5mm}~\mbox{in}~~\Omega
    \\&u=0\quad \quad\quad\quad\quad\quad\quad\quad\hspace{1mm}\hspace{0.5mm}~\mbox{on}~\partial\Omega,
 \end{aligned}
\right.
\end{equation}
where $\Omega$ is a smooth bounded domain in $\mathbb{R}^N$ for $N\geq3$, $p=(N+2)/(N-2)$, $1<q<p$, $\epsilon>0$ is a small parameter.
The existence of positive solutions to the above equation \eqref{kx} for small $\epsilon>0$ was obtained in Brezis-Nirenberg (1983) \cite{Brezis-1983} and Molle and Pistoia (2003) \cite{Molle}. However, the qualitative properties of single blow-up solutions to equation \eqref{kx}
is still unknown. Here we focus on the qualitative properties of single blow-up  solutions to the above equation \eqref{kx} for small $\epsilon$.

If $\kappa(x)=1$, by applying the various identities of derivatives of Green's function and the rescaled functions, with blow-up analysis,
we first provide a number of estimates on the
first $(N+2)$-eigenvalues and their corresponding eigenfunctions, and prove the qualitative behavior of the eigenpairs $(\lambda_{i,\epsilon}, v_{i,\epsilon})$ to the eigenvalue problem of the elliptic equation \eqref{kx} for $i=1,\cdots,N+2$. As a consequence, we have that the Morse index of a single-bubble solution is $N+1$ if the Hessian matrix of the Robin function is nondegenerate at a blow-up point.
Moreover, if $\kappa(x)\in C^2(\overline{\Omega})$, we show that, for  $\epsilon>0$ small, the asymptotic behavior of the solutions and nondegeneracy of the solutions for the problem \eqref{kx} under a nondegeneracy condition on the blow-up point of a "mixture" of both the matrix $\kappa(x)$ and Robin function.}
	\end{abstract}
	
	\vspace{3mm}
	
	\maketitle
	\section{Introduction}
\subsection{Motivation and main results}
	In this paper, we are concerned with the critical elliptic equation
\begin{equation}\label{ele-1.1-1}
\left\lbrace
\begin{aligned}
    &-\Delta u=u^{p}+\epsilon \kappa(x)u^{q} \quad \quad\mbox{in}\quad \Omega,\\
    &u>0\quad \quad \quad \quad\quad\quad\quad\quad\quad\hspace{0.5mm}  \mbox{in}\quad\Omega,\\
    &u=0\quad\quad\quad \quad\quad \quad\quad\quad\quad  \mbox{on}\hspace{2mm}\partial\Omega,
   \end{aligned}
\right.
\end{equation}
where $\Omega$ is a smooth bounded domain in $\mathbb{R}^N$ for $N\geq3$, $p=(N+2)/(N-2)$, $1\leq q<p$, $\kappa(x)\in C^2(\overline{\Omega})$ and $\epsilon>0$ is a small parameter. In the case $\kappa=-1$ and $q>p$, Merle and Peletier \cite{Merle} proved that for $N\geq3$ problem \eqref{ele-1.1-1} possesses a family of solutions concentrating at a point $x_0$, which is a critical point of the Robin function. Later, Musso and Pistoia in \cite{Musso-Pistoia-2003} proved that if $\kappa$ is a negative constant and $\Omega$ is a domain with a small hole, then problem \eqref{ele-1.1-1} has a family of solutions provided $\epsilon$ is small enough. For suitable $\kappa$, Brezis and Nirenberg \cite{Brezis-1983} proved the existence of the solutions of \eqref{ele-1.1-1} by using a variational argument  if $N\geq4$.
If $\Omega$ is star-shaped and $\kappa\equiv c\leq0$ in $\Omega$, then an application of the Pohozaev identity gives nonexistence of a nontrivial solution for \eqref{ele-1.1-1}. If $q=1$, Takahashi \cite{T} studied the qualitative
property of the blowing-up solutions to \eqref{ele-1.1-1}. Furthermore, by applying the Lyapunov-Schmidt reduction method, Molle and Pistoia in \cite{Molle} proved that $q\geq1$ if $N\geq5$, $q>1$ if $N=4$ and $q\neq p$, then there is a single-bubbling solution of \eqref{ele-1.1-1} having the form
\begin{equation}\label{pwconcentration-1}
u_{\epsilon}=PW[x_{\epsilon},\mu_{\epsilon}^{-1}\epsilon^{\frac{2}{N-6+q(N-2)}}]+R_{\epsilon},
\end{equation}
which blows-up at a point $x_0$ with the rate of the concentration $\mu_{0}^{-1}$ for $\mu_{\epsilon}^{-1}\rightarrow\mu_0^{-1}>0$, $x_{\epsilon}\rightarrow x_0\in\Omega$ and $R_{\epsilon}\rightarrow0$ in $H_{0}^1(\Omega)$ as $\epsilon\rightarrow0$. Here the function $PW[\xi,\mu]$ is the projection of $W[\xi,\mu]$ onto $H_{0}^{1}(\Omega)$, namely,
\begin{equation*}
\Delta PW[\xi,\mu]
=\Delta W[\xi,\mu],
\quad\text{in}\quad\Omega\quad
PW[\xi,\mu]=0\quad\text{on}\quad\partial\Omega,
\end{equation*}
where the bubble $W[\xi,\mu]$ with the concentration rate $\mu>0$ and the center $\xi=(\xi_1,\cdots,\xi_N)\in\mathbb{R}^N$,
\begin{equation}\label{WMIU}
W[\xi,\mu](x)=\left(\frac{\mu}{1+\big(\mu^2|x-\xi|^2/N(N-2)\big)}\right)^{\frac{N-2}{2}}\hspace{2mm}\mbox{for}\hspace{2mm}x\in\mathbb{R}^N,
\end{equation}
which are solutions of the equation
\begin{equation}\label{bec}
-\Delta W=W^{p}\quad \mbox{in}\ \ \mathbb{R}^N,\quad W>0\quad\mbox{in}\ \ \mathbb{R}^N.
\end{equation}

For the case $\kappa(x)=1$, equation \eqref{ele-1.1-1} goes back to the the
Brezis-Nirenberg problem
\begin{equation}\label{ele-1.1}
\left\lbrace
\begin{aligned}
    &-\Delta u=u^{p}+\epsilon u^{q} \quad \quad\mbox{in}\quad \Omega,\\
    &u>0\quad \quad \quad \quad\quad\quad\quad\hspace{1mm}  \mbox{in}\quad\Omega,\\
    &u=0\quad\quad\quad \quad\quad \quad\quad \hspace{0.5mm}  \mbox{on}\hspace{2mm}\partial\Omega,
   \end{aligned}
\right.
\end{equation}
Brezis and Nirenberg in a celebrated paper \cite{Brezis-1983} proved that if $N\geq4$ and $1<q<p$ problem \eqref{ele-1.1} has a solution for every $\epsilon>0$; in the case $N=3$, problem \eqref{ele-1.1} is much more delicate:
 \begin{itemize}
\item[$(a)$]
if $3<q<5$, problem \eqref{ele-1.1} has a solution for every $\epsilon>0$;
\item[$(b)$]
If $1<q\leq3$, it is only for large values of $\epsilon$ that problem \eqref{ele-1.1} possesses a solution.
\end{itemize}
In the case $q=1$, Brezis and Nirenberg \cite{Brezis-1983}
showed that if $N\geq4$, problem \eqref{ele-1.1} has a
	nontrivial solution for $\varepsilon\in(0,\lambda_1)$,  where $\lambda_1$ is the first eigenvalue of $-\Delta$ with Dirichlet boundary
	condition; when $N=3$ then there exists a constant $\lambda_{\ast}\in(0,\lambda_1)$ such that for any $\varepsilon\in(\lambda_{\ast},\lambda_1)$,  \eqref{ele-1.1} has a positive solution. If $\varepsilon=0$, problem~\eqref{ele-1.1} becomes again much more delicate. Pohozaev discovered in \cite{Pohozaev-1965} that \eqref{ele-1.1} does not have a solution if $\Omega$ is a star-shaped domain.  Bahri and Coron \cite{Bahri-1988} gave an existence result of a positive solution to problem \eqref{ele-1.1} when $\Omega$ has a
	nontrivial topology, while Ding \cite{Ding1989} constructed a solution to  \eqref{ele-1.1} in the case where $\Omega$ is contractible.
Concerning least energy solutions $u_{\epsilon}$ to \eqref{ele-1.1}, in \cite{GM} it was proved that
\begin{equation}\label{SS}
S_{\epsilon}(\Omega):=\frac{\displaystyle\int_{\Omega}|\nabla u_{\epsilon}|^2dx}{\big(\displaystyle\int_{\Omega}|u_{\epsilon}|^{p+1}dx\big)^{\frac{2}{p+1}}}\rightarrow S\quad\mbox{as}\quad\epsilon\rightarrow0,
\end{equation}
where $S$ is the best Sobolev constant defined by
	\begin{equation*}
		S:=\inf\Big\{\frac{\|\nabla u\|_{L^2}}{\|u\|_{L^{p+1}}}~:u\in \mathcal{D}^{1,2}(\mathbb{R}^N)\setminus\{0\}\Big\},
	\end{equation*}
where $\mathcal{D}^{1,2}(\mathbb{R}^N)$ denotes the closure of $C^\infty_c(\mathbb{R}^N)$ with respect to the norm $\|u\|_{\mathcal{D}^{1,2}(\mathbb{R}^N)}=\|\nabla u\|_{L^2}$.
In \cite{GM}, the authors also derived the precise asymptotic behavior of a positive least energy solution $u_{\epsilon}$ for \eqref{ele-1.1} and the exact rates of blow-up and the location of blow-up points, it was proved that a scaled $u_{\epsilon}$, given by $\|u_{\epsilon}\|_{L^{\infty}(\Omega)}u_{\epsilon}$ converges to the Green's function $G$, solution of $-\Delta G(x,\cdot)=\delta_{x}$ in $\Omega$, $G(x,0)=0$ on $\partial\Omega$, where $\delta_{x}$ denotes the Dirac measure at $x$. The location of blowing-up points are the critical points of the Robin function $\phi(x)=H(x,x)$ ( where $\phi<0$ by the maximum principle in $\Omega$), where $H(x,y)$ is the regular part of the Green function $G(x,y)$, i.e.
\begin{equation}\label{Robin}
	H(x,y)=G(x,y)-\frac{1}{(N-2)\omega_N|x-y|^{N-2}},
\end{equation}
with $\omega_N$ is the area of the unit sphere in $\mathbb{R}^N$.		

Let $\Omega$ be a smooth bounded domain in $\mathbb{R}^N$ and $N\geq4$, Rey \cite{Rey-1989} in the case $q=1$ in \eqref{ele-1.1} (independently and using different arguments, by Han in \cite{HANZCHAO}) considered
	\begin{equation}\label{eq1.2}
		\begin{cases}
			-\Delta u
			=N(N-2)u^{p}+\epsilon u
			,~~~~~~~~~~~\text{in}~~~\Omega,\\			u=0,~~~~~~~~~~~~~~~~~~~~~~~~~~~~~~~~~~~~~\quad\quad\text{on}~~\partial\Omega,
		\end{cases}
	\end{equation}
and studied the concentration of the solutions for problem~\eqref{eq1.2}. Assume that $\lbrace u_\varepsilon\rbrace$ is a minimizing sequence for best Sobolev constant $S$, then the authors established that the exact blowing-up rate of the solutions under $L^\infty$-Norm. Rey in \cite{Rey-1990} proved if $u_{\varepsilon}$ is a solution of \eqref{eq1.2} which concentrates around
	a point $x_0$ and $u_{\varepsilon}$ satisfies
	\begin{equation}\label{S1}
		|\nabla u_\epsilon|^2\rightarrow S^{\frac{N}{2}}\delta_{x_0}\quad\text{as}\quad\epsilon\rightarrow0,
	\end{equation}
	then $x_{0}\in \Omega$ is a critical point of Robin function in $\Omega$.
Also, it is nature to study the asymptotic behavior of the subcritical
solutions for the following problem
 \begin{equation}\label{SC}
 -\Delta u
=u^{p-\epsilon}
 ,~~~u>0,~~~\text{in}~~~\Omega,~~~
 		u=0,~~~\text{on}~~~\partial\Omega,
 \end{equation}
 where $\Omega$ is a bounded domain contained in $\mathbb{R}^N$ for $N\geq3$ and $\epsilon$ is a small parameter. The studies on the asymptotic behavior of radial solutions of \eqref{SC} were initiated by Atkinson and Peletier in \cite{ATKINSON-1986} by using ODE technique in the unit ball of $\mathbb{R}^3$. Later, Brezis and Peletier \cite{BP} used the method of PDE to obtain the same results as that in \cite{ATKINSON-1986} for the spherical domains. Finally, the same kind of results hold for nonspherical domain, which was settled by Han in \cite{HANZCHAO} (independently by Rey in \cite{Rey-1989}). Moreover, this result was extended in \cite{Musso-Pistoia-2002}, where Musso and Pistoia obtained the existence of multi-peak solutions for certain domains, the elliptic systems in \cite{CKIM-1} and the references therein. (See also \cite{LiWZ,Rey-1990,JW0,HL}).

Furthermore, Grossi and Pacella in \cite{GP05} deduced the asymptotic behavior of the eigenvalues $\lambda$ for the problem
 \begin{equation}\label{SC-1}
 -\Delta v
=N(N-2)(p-\epsilon)\lambda u_{\epsilon}^{p-\epsilon}v
 ,~~~v>0,~~~\text{in}~~\Omega,~~~
 		v=0,~~~\text{on}~~\partial\Omega,
 \end{equation}
and the Morse index of $u_{\epsilon}$, as well as their result in \cite{CKL-1} was extended later to the linearized problem \eqref{SC-1} at $u_{\epsilon}$ if asymptotic forms of the solutions are written as
\begin{equation*}
u_{\epsilon}=\sum_{i=1}^m\alpha_{i\epsilon}PW[\xi_{i\epsilon},\mu_{i\epsilon}\epsilon^{1/(N-2)}]+R_{\epsilon},
\end{equation*}
 where $\alpha_{i\epsilon}\rightarrow1$, $\lambda_{i\epsilon}\rightarrow\lambda_{i0}$ and $\xi_{i\epsilon}\rightarrow \xi_{i0}$ as $\epsilon\rightarrow0$. For one-bubble solutions case,  Takahashi in \cite{T-1} analyzed the linearized Brezis-Nirenberg problem of \eqref{eq1.2}. In \cite{GG09}, Gladiali and Grossi studied the asymptotic behavior of the eigenvalues and Morse index of one-bubble solution for the linearized Gelfand problem. Later that, the corresponding problem of multi-bubble solutions was proved in \cite{GGO-1}, and further the qualitative properties of the first $m$ eigenfunctions was established by Gladiali et al. in \cite{GGO}. For the studies of the related topic, see also \cite{B-L-R,dgp,LTX} and the references therein for earlier works.

Inspired by the previous work, the aim of first part of this paper is to study qualitative properties of the behavior of eigenpairs $(\lambda_{i,\epsilon},v_{i,\epsilon})$ to eigenvalue problem
\begin{equation}\label{ele-2}
\left\lbrace
\begin{aligned}
    &-\Delta v=\lambda\big(pu_{\epsilon}^{p-1}+\epsilon qu_{\epsilon}^{q-1}\big)v \quad \hspace{1mm}\mbox{in}\quad \Omega,\\
    &v=0\quad \quad \quad \quad \quad \quad \quad \quad \quad \quad \quad \hspace{1.5mm}\mbox{on}\hspace{2mm}\partial\Omega,\\&\|v\|_{L^{\infty}(\Omega)}=1
   \end{aligned}
\right.
\end{equation}
for $i=1,\cdots,N+2$ and $1<q<p$. Firstly, we concentrate on behavior of the first eigenvalue and eigenvector. Given $i=1,\cdots,N+2$,
 let $V_{i,\epsilon}$ be a dilation of $v_{i,\epsilon}$ corresponding to $\lambda_{i,\epsilon}$ defined as
\begin{equation}\label{vie}
V_{i,\epsilon}(x)=v_{i,\epsilon}\left(\big\|u_{\varepsilon}\big\|_{L^{\infty}(\Omega)}^{-\frac{p-1}{2}}x+x_{\epsilon}\right)\quad\mbox{for}\quad x\in\Omega_{\epsilon}:=\big\|u_{\epsilon}\big\|_{L^{\infty}(\Omega)}^{\frac{p-1}{2}}(\Omega-x_{\epsilon}).
\end{equation}
\begin{thm}\label{Figalli-1}
		Assume that $N\geq3$, $p=(N+2)/(N-2)$, $q\in(\max\{1,\frac{6-N}{N-2}\},p)$ and $\epsilon$ is sufficiently small. Let $\{u_{\epsilon}\}$ be a
family of least energy solutions of \eqref{ele-2}. Then
	\begin{equation}\label{vie-1-1}	
\lambda_{1,\epsilon}\rightarrow\frac{1}{p}\quad\mbox{and}\quad V_{1,\epsilon}\rightarrow W[0,1](x)=\Big(\frac{1}{1+|x|^2/N(N-2)}\Big)^{\frac{N-2}{2}}\quad\mbox{in}\quad C_{loc}^{2}(\mathbb{R}^N).
\end{equation}
Moreover, if $\epsilon$ is small enough, the eigenvalue $\lambda_{1,\epsilon}$ is simple and
\begin{equation}\label{FINAL1}	
\lim\limits_{\epsilon\rightarrow0^{+}}\big\|u_{\epsilon}\big\|^2_{L^{\infty}(\Omega)}v_{1,\epsilon}(x)=   \int_{\mathbb{R}^N}W^{p}(x)dxG(x,x_0),
\end{equation}
where the convergence is in $C^{1,\alpha}(\omega)$ with any subdomain of $\Omega$ not containing $x_0$.
\end{thm}

For $i=2,\cdots,N+1$, we have the following asymptotic behavior of the eigenfunctions $V_{i,\epsilon}$.
\begin{thm}\label{Figalli}
Let the assumptions of Theorem \ref{Figalli-1} be satisfied. Let $q\in(\max\{1,\frac{6-N}{N-2}\},p)$ if $N\geq3$, it holds that
\begin{equation}\label{vie-1}	
V_{i,\epsilon}\rightarrow\sum_{k=1}^{N}\frac{\alpha_k^ix_k}{(N(N-2)+|x|^2)^{\frac{N}{2}}}\quad\mbox{as}\hspace{2mm}\epsilon\rightarrow0\hspace{2mm}\mbox{in}\hspace{2mm}C_{loc}^{1}(\mathbb{R}^N),\hspace{2mm}i=2,\cdots,N+1.
\end{equation}
If the parameters $N$ and $q$ are chosen in the following range
\begin{equation}\label{qN}
\left\lbrace
\begin{aligned}
    &1<q<p\hspace{2.7mm}\mbox{if}\hspace{2mm}N\geq5,\\
    &\frac{3}{2}<q<p \hspace{2mm}\mbox{if}\hspace{2mm} N=4,\\
    &3<q<5\hspace{2.7mm}\mbox{if}\hspace{2mm}N=3,
   \end{aligned}
\right.
\end{equation}
it holds that
\begin{equation}\label{vie-2}	
\big\|u_{\epsilon}\big\|_{\infty}^{\frac{p+3}{2}}v_{i,\epsilon}\rightarrow\frac{a_N}{N-2}\omega_N\sum_{k=1}^{N}\alpha_{k}^{i}\left(\frac{G\left(x,y\right)}{\partial y_k}\right)\big|_{y=x_0}
\quad\mbox{as}\hspace{2mm}\epsilon\rightarrow0\hspace{2mm}\mbox{in}\hspace{2mm}C_{loc}^{1}(\overline{\Omega}\setminus\{x_0\}),
\end{equation}	
for $i=2,\cdots,N+1$. Here the vectors  $\alpha^{i}=(\alpha_1^{i},\cdots,\alpha_{N}^i)\neq0$ in $\mathbb{R}^N$ and some suitable constant $a_N$ and satisfying
$$a_N=\int_{0}^{\infty}\frac{r^{N-1}dr}{(N(N-2)+r^2)^{\frac{N+2}{2}}}.
$$
\end{thm}
\begin{Rem}\label{qdefan}
Since $q>1$ exerts a significant influence in study of problem \eqref{ele-2}, especially the obvious differences observed from the comparison between $q$ and the exponent $\frac{N}{N-2}$, we have not obtained conclusion \eqref{vie-2} in the case $1<q\leq\frac{3}{2}$ if $N=4$.
\end{Rem}
	For a further understanding of eigenvalue $\lambda_{i,\epsilon}$, we establish a connection between the eigenvalues $\lambda_{i,\epsilon}$ and the eigenvalues of the Hessian matrix of the Robin function $\phi(x)$ for any $i=2,\cdots,N+1$. Our result can be stated as follows.
	\begin{thm}\label{remainder terms}
		Let the assumptions of Theorem \ref{Figalli-1} be satisfied. Assume that $\nu_1\leq\nu_2\leq\cdots\leq\nu_N$ the eigenvalues of the Hassian matrix $D^2\phi(x_0)$ of the Robin function at $x_0$.
		Then there exist some constant $\mathcal{M}>0$ such that for $\epsilon$ small, it holds that
\begin{equation*}
(\lambda_{i,\epsilon}-1)\big\|u_{\epsilon}\big\|_{\infty}^{p+1}\rightarrow \mathcal{M}\nu_{i-1} \quad\mbox{for any}\hspace{2mm}i=2,\cdots,N+1,
		\end{equation*}
with $$\mathcal{M}:=\frac{(N(N-2))^{\frac{N}{2}}a_{N}\omega_N^2}{p\displaystyle\int_{\mathbb{R}^N}W^{p-1}[0,1]|\nabla W^2[0,1]|dx}>0.$$
Moreover, the vectors $\alpha^i$ of \eqref{vie-1} are the eigenvectors corresponding to $\nu_i$ for $i=2,\cdots,N+1.$
	\end{thm}
Now we have the folowing the qualitative properties of the nodal regions of the eigenfunctions.
\begin{thm}\label{eigenfunctions-1}
Let the assumptions of Theorem \ref{Figalli-1} be satisfied. Define
$N_{i,\epsilon}=\{x\in\Omega:~v_{i,\epsilon}(x)=0\}.$
Then
\begin{itemize}
\item[$(a)$]
$N_{i,\epsilon}\cap\partial\Omega=\emptyset$ if $\Omega$ be convex, for $i=2,\cdots,N+1$.
\item[$(b)$] the eigenfunctions $v_{i,\epsilon}(x)$ have only two nodal regions for $i=2,\cdots,N+1.$
\end{itemize}
    \end{thm}

When $\Omega$ is a domain convex in the $x_j$ direction and symmetric to hyperplanes $\mathcal{T}_{j}=\{x\in\mathbb{R}^N,~x_j=0\}$ for $j=1,\cdots,N$.
We will show the following:
\begin{thm}\label{eigenfunctio1}
Let
$$\mathscr{D}_{j}^{-}=\big\{x\in\Omega:~x_{j}(x)<0,~j=1,\cdots N\big\}.$$
Then for $\epsilon$ sufficiently small each eigenvalue $\lambda_{i,\epsilon}$ is the first eigenvalue $\mu_{j,\epsilon}^{(1)}$ of the eigenvalue problem
\begin{equation}\label{el}
\left\lbrace
\begin{aligned}
    &-\Delta v=\mu\big(pu_{\epsilon}^{p-1}+\epsilon qu_{\epsilon}^{q-1}\big)v \quad \hspace{1mm}\mbox{in}\quad \mathscr{D}_{j}^{-},\\
    &v=0\quad \quad \quad \quad \quad \quad \quad \quad \quad \quad \quad \hspace{2mm}\mbox{on}\hspace{2mm}\partial \mathscr{D}_j^{-},
   \end{aligned}
\right.
\end{equation}
for $i=1,\cdots,N+2$. Moreover the eigenfunction $v_{i,\epsilon}$ corresponding to $\mu_{i,\epsilon}$ is odd in the $x_j$-variable and even with respect to the other variables.
  \end{thm}

It is natural to investigate the $(N+2)$-th eigenvalue of \eqref{ele-2} and some qualitative properties of the corresponding eigenfunction. Then we establish the following results.
	\begin{thm}\label{thmprtb}
		Let the assumptions of Theorem \ref{Figalli-1} be satisfied.
		Then we have
		\begin{equation}\label{rtbdy46}	
V_{N+2,\epsilon}\rightarrow \beta \frac{(N(N-2)-|x|^2}{(N(N-2)+|x|^2)^{\frac{N}{2}}}\quad\mbox{as}\hspace{2mm}\epsilon\rightarrow0\hspace{2mm}\mbox{in}\hspace{2mm}C_{loc}^{1}(\overline{\Omega}\setminus\{x_0\})
\end{equation}
with $\beta\neq0$ and there exist a constant $\mathcal{C}_{p,q}$ such that
	\begin{equation}\label{rt}	
1-\lambda_{N+2,\epsilon}=\frac{1}{\|u_{\epsilon}\|_{L^{\infty}(\Omega)}^2}(\mathcal{C}_{p,q}+o(1))
\end{equation}
with
$$
\mathcal{C}_{p,q}=\frac{(N-2)R_N|\phi(x_0)|+\frac{2(p-q)}{p-1}B_{p,q}|\phi(x_0)|\mathcal{E}_N}{\mathcal{F}_N}<0,$$
where $\mathcal{F}_N=\frac{\pi^{\frac{N}{2}}}{2^{N+1}}$, $\mathcal{E}_{N}$ is a constant (see section \ref{section6}), $B_{p,q}$ and $R_N$ are defined as follows
$$B_{p,q}:=\frac{2(q+1)}{p-q+1}\frac{(N(N-2))^{\frac{N}{2}}\omega_N}{N^2}\frac{\Gamma(\frac{(N-2)(q+1)}{2})}{\Gamma(\frac{N}{2})\Gamma(\frac{(N-2)(q+1)-N}{2})},$$
$$R_N:=\frac{p\omega_N^2}{2N}(N(N-2))^{\frac{N}{2}}\Big(\frac{\Gamma(\frac{N}{2})\Gamma(2)-\Gamma(\frac{N}{2}+1)\Gamma(1)}{\Gamma(N/2+2)}\Big).$$
\end{thm}
\begin{thm}\label{thmprtb-1}
The eigenvalue $\lambda_{N+2,\epsilon}$ is simple and corresponding eigenfunction $v_{N+2,\epsilon}$ has only two nodal regions if $\epsilon$ is small enough. Moreover, the closure of the nodal set of the eigenfunction $v_{N+2,\epsilon}$ does not touch the boundary of $\Omega$.
\end{thm}
	
Let us recall that the Morse index of a solution $u_{\epsilon}$ to \eqref{ele-1.1} can be defined as
$$m(u_{\epsilon}):=\sharp\{k\in\mathbb{N}:~\lambda_{k,\epsilon}<1\}$$
where $\lambda_{1,\epsilon}<\lambda_{2,\epsilon}\leq\lambda_{3,\epsilon},\cdots$ is the sequence of eigenvalues for the linearized problem of \eqref{ele-1.1}.
Also the augmented Morse index of $x_0$, as a critical point of $\phi(x)$:
$$m(x_0):=\sharp\{k\in\mathbb{N}:~\nu_{k}\leq0\},$$
where $\nu_1\leq\nu_2\leq\cdots\nu_N$ are the eigenvalues of the Hessian matrix $D^2\phi(x_0)$.
As a result, we obtain the following corollary.
\begin{cor}\label{emm-1}
Let the assumptions of Theorem \ref{Figalli-1} be satisfied. Then
$$1\leq1+m(x_0)\leq m(u_{\epsilon})\leq N+1$$
for sufficiently small $\epsilon>0$. Therefore if the matrix $D^2\phi(x_0)$ nondegenerate, then we have
$$m(u_{\epsilon})=1+m(x_0).$$
\end{cor}	

On the other hand, we naturally also consider the asymptotic behavior and nondegeneracy of the solutions of \eqref{ele-1.1-1} in the case $N\geq3$, $1<q<p$ and $\kappa(x)\in C^2(\overline{\Omega})$. Our result can be stated as follows.
\begin{thm}\label{Figa}
Let $1<q<p$ if $N\geq5$, $\frac{3}{2}<q<p$ if $N=4$, $\epsilon$ is sufficiently small, $\mathscr{K}_{+}:=\{x\in\Omega:~\kappa(x)>0\}\neq\emptyset$. Assume that $\{u_{\epsilon}\}$ be a
family of solutions of \eqref{ele-1.1-1} which blow-up and concentrate at a point $x_0\in\Omega$. If the matrix
\begin{equation}\label{CC}
\left(\frac{1}{\kappa}\big(\frac{\partial^2 \kappa}{\partial x_{i}\partial x_{j}}\big)-\Gamma_{p,q}\frac{1}{\phi}\big(\frac{\partial^2\phi}{\partial x_{i}\partial x_{j}}\big)\right)_{1\leq i,j\leq N}(x)\quad \mbox{in}\quad\mathscr{K}_{+}
\end{equation}
is non singular at $x_0$, then $u_{\epsilon}$ is nondegenerate, i.e. the linearized problem
\begin{equation}\label{nondegene}
\left\lbrace
\begin{aligned}
    &-\Delta w=\big(pu^{p-1}+\epsilon q\kappa(x)u^{q-1}\big)w \quad \quad\mbox{in}\quad \Omega,\\
    &w=0\quad\quad\quad \quad\quad\quad\quad\quad\quad\quad\quad\quad\quad\hspace{3mm} \mbox{on}\hspace{2mm}\partial\Omega,
   \end{aligned}
\right.
\end{equation}
admits only the trivial solution $w\equiv0$. $\Gamma_{p,q}$ is given by
$$\Gamma_{p,q}=\frac{a_N(p-q+1)N}{2b_N}\frac{\Gamma(\frac{N}{2})\Gamma(\frac{(N-2)(q+1)-N}{2})}{\Gamma(\frac{(N-2)(q+1)}{2})}$$
with
$$a_N=\int_{0}^{\infty}\frac{r^{N-1}dr}{(N(N-2)+r^2)^{\frac{N+2}{2}}}
\quad\mbox{and}\quad b_N=\int_{0}^{\infty}\frac{r^{N-1}dr}{(N(N-2)+r^2)^{\frac{N-2}{2}(q+1)}}.$$
\end{thm}
\begin{Rem}
As noted in Remark \ref{qdefan} earlier, we cannot guarantee the validity of the above conclusion for the case $1<q\leq\frac{3}{2}$ and $N=4$.
\end{Rem}

\subsection{Structure of the paper.}
	The paper is organized as follows. In section \ref{section2}, we establish a decay estimate of rescaled eigenfunctions $V_{i,\epsilon}$. Section \ref{asymptotic} is devoted to consider the first eigenvalue and the first eigenfunction, and concluding the proof of Theorem \ref{Figalli-1}. In sections \ref{section4}-\ref{section3-1}, we give some limit characterizations
of the eigenfunctions $v_{i,\epsilon}$ and an important estimate on the eigenvalues $\lambda_{i,\epsilon}$ for $i=2,\cdots,N+1$, respectively. Based on these results and combined with the various identities of derivatives of Green's function and the rescaled functions, we complete the proof of Theorem \ref{Figalli} in section \ref{section6}.
In section \ref{sangshen}, Theorem \ref{remainder terms} is proved and by following the same argument as in \cite{GP05}, we give the proof of Theorems \ref{eigenfunctions-1} and \ref{eigenfunctio1} in section \ref{section9}. In section \ref{section8}, we establish an important estimate on the eigenvalue $\lambda_{N+2,\epsilon}$ and we conclude that the proof of Theorems \ref{thmprtb} and \ref{thmprtb-1}, and using those results, we have that, the Morse
index of the solution $u_{\epsilon}$ is $N+1$. Motivated by the framework and methods of the proof of Theorem \ref{Figalli} and Theorem \ref{remainder terms}, in the last section we will continue to study the nondegeneracy result of solutions of \eqref{ele-1.1-1} and we prove Theorem \ref{Figa}.

\subsection{Notations.}
Throughout this paper, $C$ and $c$ are indiscriminately used to denote various absolutely positive constants. $a\approx b$ means that $a\lesssim b$ and $\gtrsim b$, and we will use big $O$ and small $o$ notations to describe the limit behavior of a certain quantity as $\epsilon\rightarrow0$.

\section{Decay estimate of the rescaled eigenfunction $V_{i,\epsilon}$}\label{section2}
In this section, we are devoted to establish a decay estimate of rescaled least energy solutions $v_{i,\epsilon}$ for $i\in\mathbb{N}$. For the simplicity of notations, we write $W(x)$ instead of $W[0,1](x)$ and $\|\cdot\|_{\infty}$ instead of the norm of $L^{\infty}(\Omega)$ in the sequel. first of all, we recall the analysis of Gao et al. in \cite{GM} concerning the least energy solutions $u_{\epsilon}$ to the problem \eqref{ele-1.1}.
We choose a point $x_\epsilon\in\Omega$ and a number $\lambda_{\epsilon}>0$ by the relation
\begin{equation}\label{miu}
\mu_{\epsilon}^{\frac{2}{p-1}}=\big\|u_{\epsilon}\big\|_{\infty}=u_\epsilon(x_\epsilon).
\end{equation}
Then one has that $\mu_{\epsilon}\rightarrow\infty$ as $\epsilon\rightarrow0$ and due to Lemma 2.5 in \cite{GM},
\begin{equation}\label{2-miu}
\epsilon\lesssim \mu_{\epsilon}^{-\frac{N-2}{2}(q-p+2)}=\big\|u_{\epsilon}\big\|_{\infty}^{p-q-2}.
\end{equation}
 We define a family of rescaled functions
\begin{equation}\label{miu-1}
U_{\epsilon}(x)=\mu_{\epsilon}^{-\frac{2}{p-1}}u_{\epsilon}(\mu_{\epsilon}^{-1}x+x_{\epsilon})\quad\mbox{for}\quad x\in\Omega_{\epsilon}=\mu_{\epsilon}(\Omega-x_{\epsilon}).
\end{equation}	
Granted these notions, we have
\begin{thm}[Theorem 1.1 of \cite{GM}]
Assume that $N\geq3$, $q\in(\max\{1,\frac{6-N}{N-2}\},2^{\ast}-1)$ and $\epsilon$ is sufficiently small. Let $u_\epsilon$ be a solution of \eqref{ele-1.1} satisfying \eqref{SS}.
 Then if $x_{\epsilon}$ is a maximum point $u_{\epsilon}$, i.e. $u_{\epsilon}(x_{\epsilon})=\|u_{\epsilon}\|_{L^{\infty}(\Omega)}$, we have that $x_{\epsilon}\rightarrow x_0\in\Omega$ as $\epsilon\rightarrow0$ and the following results hold:
\begin{itemize}
\item[$(a)$]
$x_0$ is a critical of $\phi(x)$ (Robin's function of $\Omega$) for $x\in\Omega$, and
\begin{equation}\label{Fn}
\lim\limits_{\epsilon\rightarrow0}\epsilon\big\|u_\epsilon\big\|_{\infty}^{q+2-p}=B_{p,q}\phi(x_0),
\end{equation}
where $$B_{p,q}=\frac{2(q+1)}{p-q+1}\frac{(N(N-2))^{\frac{N}{2}}\omega_N}{N^2}\frac{\Gamma(\frac{(N-2)(q+1)}{2})}{\Gamma(\frac{N}{2})\Gamma(\frac{(N-2)(q+1)-N}{2})}.$$
\item[$(b)$] For any $x\in\Omega\setminus\{x_0\}$, it holds that
\begin{equation}\label{ifini}
\lim\limits_{\epsilon\rightarrow0^{+}}\big\|u_{\epsilon}\big\|_{\infty}u_{\epsilon}(x)=\frac{1}{N}(N(N-2))^{\frac{N}{2}}\omega_NG(x,x_{0})\hspace{2mm}\mbox{in} \hspace{2mm}C^{1,\alpha}(\Omega\setminus\{x_0\}).
\end{equation}
\item[$(c)$] There exists a positive constant $C$, independent of $\epsilon$, such that
\begin{equation}\label{dus}
u_{\epsilon}(x)\leq C\big\|u_{\epsilon}\big\|_{\infty}\Big(\frac{N(N-2)}{N(N-2)+\|u_{\epsilon}\|^{p-1}_{\infty}|x-x_{\epsilon}|^2}\Big)^{\frac{N-2}{2}}.
\end{equation}
\end{itemize}
\end{thm}
We first state the following lemma which is useful in our analysis (see \cite{WY} for the proof).
\begin{lem}\label{Lem6.1}
For any constant $0<\sigma<N-2$, there is a constant $C>0$ such that
\begin{equation}\label{Lem6.1-0}
\int_{\mathbb{R}^n}\frac{1}{|y-x|^{N-2}}\frac{1}{(1+|x|)^{2+\sigma}}dx\leq
\frac{C}{(1+|y|)^\sigma}.
\end{equation}
\end{lem}
We conduct a decay estimate for solutions of the eigenvalue problem \eqref{ele-2}.
\begin{lem}
Assume that $v_{i,\epsilon}$ is a solution of \eqref{ele-2}, $i\in\mathbb{N}$. Then there exists a positive constant $C$, independent of $\varepsilon$, such that
\begin{equation}\label{decay22}
v_{i,\epsilon}(x)\leq CW\Big(\frac{N(N-2)}{N(N-2)+\|u_{\epsilon}\|^{p-1}_{\infty}|x-x_{\epsilon}|^2}\Big)^{\frac{N-2}{2}}\quad \mbox{for}\quad x\in\mathbb{R}^N,
\end{equation}
for each point $x\in\Omega$, and there holds
\begin{equation}\label{day2}
|V_{i,\epsilon}(x)|\leq C\left(\frac{N(N-2)}{N(N-2)+|x|^2}\right)^{\frac{N-2}{2}}\quad \mbox{for}\quad x\in\mathbb{R}^N,
\end{equation}
where $\tilde{v}_{i,\epsilon}$ can be found in \eqref{vie}.	
\end{lem}	
\begin{proof}
We have
\begin{equation}\label{e-1}
\left\lbrace
\begin{aligned}
    &-\Delta v_{i,\epsilon}=\lambda_{i,\epsilon}\big(pu_{\epsilon}^{p-1}+\epsilon qu_{\epsilon}^{q-1}\big)v_{i,\epsilon} \quad \mbox{in}\quad \Omega,\\
    &v_{i,\epsilon}=0\quad \quad \quad \quad \quad \quad \quad \quad \quad \quad \quad \hspace{6.5mm} \mbox{on}\hspace{2.5mm}\partial\Omega,\\&\|v_{i,\epsilon}\|_{\infty}=1.
   \end{aligned}
\right.
\end{equation}
Then we get from the integral representation of $v_{i,\epsilon}$
that
\begin{equation*}
v_{i,\epsilon}(x)=\lambda_{i,\epsilon}
\Big[p\int_{\Omega}u_{\epsilon}^{p-1}v_{i,\epsilon}G(x,y)dy+\epsilon q\int_{\Omega} u_{\epsilon}^{q-1}v_{i,\epsilon}G(x,y)dy\Big].
\end{equation*}
Combining \eqref{dus} and \eqref{Lem6.1-0}-\eqref{decay22}, we deduce that
\begin{equation*}
\begin{split}
&p\int_{\Omega}u_{\epsilon}^{p-1}v_{i,\epsilon}G(x,y)dy+\epsilon q\int_{\Omega} u_{\epsilon}^{q-1}v_{i,\epsilon}G(x,y)dy\\&=O\Big(\int_{\Omega}\frac{1}{|x-y|^{N-2}} W^{p-1}\Big(\big\|u_{\epsilon}\big\|_{\infty}^{\frac{p-1}{2}}(x-x_{\epsilon})\Big)dy
+\epsilon\int_{\Omega}\frac{1}{|x-y|^{N-2}} W^{q}\Big(\big\|u_{\epsilon}\big\|_{\infty}^{\frac{p-1}{2}}(x-x_{\epsilon})\Big)dy\Big)\\&
\\&= O\Big(\frac{N(N-2)}{(N(N-2)+\|u_{\epsilon}\|_{\infty}^{(p-1)/2}|x-x_{\epsilon}|)^2}\Big)
+O\Big(\frac{\epsilon}{(N(N-2)+\|u_{\epsilon}\|_{\infty}^{(p-1)/2}|x-x_{\epsilon}|)^{(n-2)q-2}}\Big),
\end{split}
\end{equation*}
As a consequence,
$$v_{i,\epsilon}= O\Big(\frac{(N(N-2)}{(N(N-2)+\|u_{\epsilon}\|_{\infty}^{(p-1)/2}|x-x_{\epsilon}|)^{2}}\Big)+O(\epsilon).$$
Next repeating the above process, we know
\begin{equation*}
v_{i,\epsilon}= O\Big(\frac{(N(N-2)}{(N(N-2)+\|u_{\epsilon}\|_{\infty}^{(p-1)/2}|x-x_{\epsilon}|)^{4}}\Big)+O(\epsilon^2).
\end{equation*}
Finally, we can proceed as in the above argument for finite number of times to  conclude that \eqref{decay22}. By scaling, we get \eqref{day2}.
\end{proof}
It is noticing that \eqref{dus} is equivalent to the estimate
\begin{equation}\label{U-day2}
|U_{\epsilon}(x)|\leq CW(x)\quad \mbox{for}\quad x\in\mathbb{R}^n.
\end{equation}
Hence combing \eqref{day2} and the standard elliptic regularity, we obtain the following result.
\begin{lem}\label{UV1}
It holds
\begin{equation}\label{UV}
U_{\epsilon}\rightarrow W(x)=\Big(\frac{1}{1+|x|^2/N(N-2)}\Big)^{\frac{N-2}{2}}\hspace{2mm}\mbox{in}\hspace{2mm} C_{loc}^{2}(\mathbb{R}^N)
\end{equation}
and there exists a $v_i$ such that $V_{i,\epsilon}\rightarrow v_{i}$ in $C_{loc}^{2}(\mathbb{R}^N)$,$i\in\mathbb{N}$ where $v_{i}$ solves \eqref{wp11}.
\end{lem}
\begin{proof}
We have that $0<V_{i,\epsilon}(x)\leq1$ for $x\in\Omega_{\epsilon}$ and $V_{i,\epsilon}$ satisfies
\begin{equation}\label{Wtilta}
\left\lbrace
\begin{aligned}
    &-\Delta V_{i,\epsilon}=\big(pU_{\epsilon}^{p-1}+\frac{\epsilon q}{\|u_{\epsilon}\|_{\infty}^{p-q}}U_{\epsilon}^{q-1}\big)V_{i,\epsilon} \quad \mbox{in}\quad \Omega,\\
    &V_{i,\epsilon}=0\quad \quad \quad\quad\quad\quad\quad \quad \quad \quad \quad \quad \quad \quad \quad \hspace{4mm} \mbox{on}\hspace{2.5mm}\partial\Omega,\\&\|V_{i,\epsilon}\|_{\infty}=1.
   \end{aligned}
\right.
\end{equation}
Thus there exist a function $v_{i}\in \mathcal{D}^{1,2}(\mathbb{R}^N)$ such that $V_{i,\epsilon}\rightharpoonup v_{i}$ in $\mathcal{D}^{1,2}(\mathbb{R}^N)$ as $\epsilon\rightarrow0$. Next by standard elliptic regular theory, we have $V_{i,\epsilon}\rightarrow v_{i}$ in $C_{loc}^{2}(\mathbb{R}^N)$. Passing to the limit in \eqref{Wtilta} we get $v_{i}$ satisfies \eqref{wp11}. In same argument, we can deduce that $U_{\epsilon}\rightarrow W(x)$ by combining the elliptic estimate and uniqueness theorem of solutions.
\end{proof}

\section{Asymptotic behavior of the eigenpair $(\lambda_{1,\epsilon},v_{1,\epsilon})$}\label{asymptotic}
In this section, we are devoted to prove Theorem \ref{Figalli-1}.
Before proving the main theorem, we need the following result.
\begin{Prop}\label{prondgr}
    Let $\lambda_i$, for $i=1,2,\ldots,$ denote the eigenvalues of the
     following eigenvalue problem
    \begin{equation}\label{wp11}
    -\Delta v_i
    =\lambda_ipW^{p-1}v_i\quad \mbox{in}\quad \mathbb{R}^N
    \end{equation}
for all $v_i\in \mathcal{D}^{1,2}(\mathbb{R}^N)$. Then
$\lambda_1=\frac{1}{p}$ is simple and the corresponding eigenfunction is $W(x)$, and $\lambda_{2}=\lambda_{3}=\cdots=\lambda_{N+2}=1$ with the corresponding $(N+1)$-dimensional eigenfunction space spanned by in $\mathcal{D}^{1,2}(\mathbb{R}^N)$ of the form
\begin{equation*}
\frac{x_1}{(N(N-2)+|x|^2)^{\frac{N}{2}}},\cdots,\frac{x_N}{(N(N-2)+|x|^2)^{\frac{N}{2}}}\quad\mbox{and}\quad\frac{N(N-2)-|x|^2}{(N(N-2)+|x|^2)^{\frac{N}{2}}}.
\end{equation*}
    Furthermore $\lambda_1<\lambda_{2}=\lambda_{3}=\cdots=\lambda_{N+2}<\lambda_{N+3}\leq\cdots$.
    \end{Prop}
	\begin{proof}
		For the derivation of conclusion are found in \cite{BE91}.
	\end{proof}
    The following result will be used to complete the proof of Theorem \ref{Figalli-1}.
    \begin{lem}[\cite{HANZCHAO}]\label{regular}
Let $u$ solve
\begin{equation*}
\begin{cases}
-\Delta u=f\quad\mbox{in}\hspace{2mm}\Omega\subset\mathbb{R}^N,\\
u=0\quad\quad\hspace{3mm}\mbox{on}\hspace{2mm}\partial\Omega.
\end{cases}
\end{equation*}
Then
\begin{equation}\label{regu}
\|u\|_{W^{1,r}(\Omega)}+\|\nabla u\|_{C^{0,\alpha}(\omega^{\prime})}\leq C\big(\|f\|_{L^1(\Omega)}+\|f\|_{L^\infty(\omega)}\big)
\end{equation}
for $r<n/(n-1)$ and $\alpha\in(0,1)$. Here $\omega$ be a neighborhood of $\partial\Omega$ and $\omega^{\prime}\subset\omega$ is a strict subdomain of $\omega$.
\end{lem}
Now we are ready to prove Theorem \ref{Figalli-1}.	
\begin{proof}[Proof of Theorem \ref{Figalli-1}]
By the variational
characterization of the eigenvalue $\lambda_{1,\epsilon}$, we have
\begin{equation*}
\begin{split}
			\lambda_{1,\epsilon}&=\inf\limits_{v\in H_{0}^1(\Omega)\setminus\{0\}}\frac{\displaystyle\int_{\Omega}|\nabla v(x)|^2dx}{\displaystyle\int_{\Omega}\big(pu_{\epsilon}^{p-1}+\epsilon qu_{\epsilon}^{q-1}\big)v^2 dx}.
\end{split}
		\end{equation*}
Taking $v=u_{\epsilon}$, then we get
\begin{equation*}
\begin{split}
			\lambda_{1,\epsilon}&\leq\frac{\displaystyle\int_{\Omega}|\nabla u_{\epsilon}(x)|^2dx}{p\displaystyle\int_{\Omega}u_{\epsilon}^{p+1}dx+\epsilon q\displaystyle\int_{\Omega}u_{\epsilon}^{q+1}dx}\\&=\frac{\displaystyle\int_{\Omega_{\epsilon}}U_{\epsilon}^{p+1}dx+\epsilon \|u_{\epsilon}\|_{\infty}^{-(p-q)}\displaystyle\int_{\Omega_{\epsilon}}U_{\epsilon}^{q+1}dx}{p\displaystyle\int_{\Omega_{\epsilon}}U_{\epsilon}^{p+1}dx+\epsilon q\|u_{\epsilon}\|_{\infty}^{-(p-q)}\displaystyle\int_{\Omega_{\epsilon}}U_{\epsilon}^{q+1}dx}
=\frac{\displaystyle\int_{\mathbb{R}^N} W^{p+1}(x)dx+o(1)}{p\displaystyle\int_{\mathbb{R}^N}W^{p+1}(x)+o(1)}
\end{split}
		\end{equation*}
as $\epsilon\rightarrow0$, which means that
$$\limsup\limits_{\epsilon\rightarrow0}\lambda_{1,\epsilon}\leq\frac{1}{p}.$$
Then up to a subsequence we have $\lambda_{1,\epsilon}\rightarrow\lambda_1\in[0,1/p]$ for $\epsilon$ sufficiently small.
Moreover by Lemma \ref{identity}, up to a subsequence, there exists a function $v_0$, such that we have $V_{1,\epsilon}\rightarrow v_0$ in $C_{loc}^{1}(\mathbb{R}^n)$, where $v_0$ satisfies
\begin{equation*}
\left\lbrace
\begin{aligned}
    &-\Delta v_{0}=\lambda\big(pW^{p-1}+\epsilon qW^{q-1}\big)v_{0} \quad \mbox{in}\quad \mathbb{R}^N,\\
    &\int_{\mathbb{R}^N}|\nabla v_{0}|^2dx<\infty,\quad\|v_{0}\|_{L^{\infty}(\mathbb{R}^N)}=1.
   \end{aligned}
\right.
\end{equation*}
It follows from Proposition \ref{prondgr} that
\begin{equation}\label{e-00-1}
\lambda_1=\frac{1}{p}\quad\mbox{and}\quad v_0=W(x)=\Big(\frac{1}{1+|x|^2/N(N-2)}\Big)^{\frac{N-2}{2}},
\end{equation}
and we conclude the result \eqref{vie-1-1}.

Now we will show that $\lambda_{1,\epsilon}$ is simple. Then we may assume by contradiction that there exist at least two eigenfunctions $v_{1,\epsilon}^{(1)}$ and $v_{1,\epsilon}^{(2)}$ corresponding to $\lambda_{1,\epsilon}$ orthogonal in the space $H_{0}^{1}(\Omega)$, and by \eqref{vie-1-1} and \eqref{UV1} with the scaling, as $\epsilon$ small enough we derive that
\begin{equation*}
		\begin{split}	
\int_{\Omega}\big(pu_{\epsilon}^{p-1}+\epsilon qu_{\epsilon}^{q-1}\big)v_{1,\epsilon}^{(1)}v_{1,\epsilon}^{(2)}dx =0
\Rightarrow\int_{\mathbb{R}^N}W^{p+1}(x)dx=0,
\end{split}
\end{equation*}
a contradiction.

Finally, we prove \eqref{FINAL1}. We find
\begin{equation}\label{laplacian}
\begin{split}
-\Delta\big(\big\|u_{\epsilon}\big\|^2_{\infty}v_{1,\epsilon}\big)=&\lambda_{1,\epsilon}\big\|u_{\epsilon}\big\|^2_{\infty}\big(pu_{\epsilon}^{p-1}+\epsilon qu_{\epsilon}^{q-1}\big)v_{1,\epsilon}\quad\mbox{in}\quad \Omega.
\end{split}
\end{equation}
We integrate the right-hand side of \eqref{laplacian},
\begin{equation}\label{right}
\begin{split}
\lambda_{1,\epsilon}\big\|u_{\epsilon}\big\|^2_{\infty}\int_{\Omega}\big(pu_{\epsilon}^{p-1}+\epsilon qu_{\epsilon}^{q-1}\big)v_{1,\epsilon}=\lambda_{1,\epsilon}\int_{\Omega}\big(pU_{\epsilon}^{p-1}+\epsilon q\big\|u_{\epsilon}\big\|^{-(p-q)}_{\infty}U_{\epsilon}^{q-1}\big)V_{1,\epsilon}dy
\end{split}
\end{equation}
Therefore, combining \eqref{vie-1-1}, \eqref{U-day2} and \eqref{e-00-1} by dominated convergence, we obtain
\begin{equation*}
\begin{split}
\lim\limits_{\epsilon\rightarrow0}\int_{\Omega}\lambda_{1,\epsilon}\big\|u_{\epsilon}\big\|^2_{\infty}\int_{\Omega}\big(pu_{\epsilon}^{p-1}+\epsilon qu_{\epsilon}^{q-1}\big)v_{1,\epsilon}&=\int_{\mathbb{R}^N}W^{p}(x)dx<\infty.
\end{split}
\end{equation*}
Also using the bound \eqref{2-miu}, \eqref{U-day2} and \eqref{e-00-1} ,
we have
\begin{equation*}
\begin{split}
\lambda_{1,\epsilon}\big\|u_{\epsilon}\big\|^2_{\infty}\big(pu_{\epsilon}^{p-1}+\epsilon qu_{\epsilon}^{q-1}\big)v_{1,\epsilon}\leq
\frac{M_1}{\big\|u_{\epsilon}\big\|^{p-1}_{\infty}}\frac{1}{\big|x-x_{0}\big|^{(N-2)p}}
+\big(\frac{1}{\big\|u_{\epsilon}\big\|_{\infty}}\big)^{\frac{(N-2)(q-1)q}{2}-p+1}\frac{M_2}{\big|x-x_{0}\big|^{(N-2)q}}
\end{split}
\end{equation*}
for $x\neq x_0$ and some $M_1,~M_2>0$.
But $\frac{(N-2)(q-1)q}{2}-p+1>0$ since $q>1$, one has
$$\lambda_{1,\epsilon}\big\|u_{\epsilon}\big\|^2_{\infty}\big(pu_{\epsilon}^{p-1}+\epsilon qu_{\epsilon}^{q-1}\big)v_{1,\epsilon}\rightarrow0\hspace{2mm}\mbox{for}\hspace{2mm} x\neq x_0.$$
From here and \eqref{laplacian} we deduce that
$$
-\Delta(\big\|u_{\epsilon}\big\|^2_{\infty}v_{1,\epsilon})\rightarrow\int_{\mathbb{R}^N}W^{p}(x)dx\delta_{x=x_0}
$$
in the sense of distributions in $\Omega$ as $\epsilon\rightarrow0$. Let $\omega$ be any neighborhood of $\partial\Omega$ not containing $x_0$. By regularity theory, see Lemma \ref{regular}, we obtain
\begin{equation*}
\begin{split}
\Big\|\|u_{\epsilon}\|^2_{\infty}v_{1,\epsilon}\Big\|_{C^{1,\alpha}(\omega^{\prime})}\lesssim& \Big\|\lambda_{1,\epsilon}\big\|u_{\epsilon}\big\|^2_{\infty}\big(pu_{\epsilon}^{p-1}+\epsilon qu_{\epsilon}^{q-1}\big)v_{1,\epsilon}\Big\|_{L^1(\Omega)}\\&+\Big\|\lambda_{1,\epsilon}\big\|u_{\epsilon}\big\|^2_{\infty}\big(pu_{\epsilon}^{p-1}+\epsilon qu_{\epsilon}^{q-1}\big)v_{1,\epsilon}\Big\|_{L^\infty(\omega)}.
\end{split}
\end{equation*}
As a result,
$$ \big\|u_{\epsilon}\big\|^2_{\infty}v_{1,\epsilon}\rightarrow\int_{\mathbb{R}^N}W^{p}(x)dxG(x,x_0)
\hspace{2mm}\mbox{in}\hspace{2mm}C^{1,\alpha}(\omega)\hspace{2mm}\mbox{as}\hspace{2mm}\epsilon\rightarrow0
$$
for any subdomain $\omega$ of $\partial\Omega$ not containing $x_0$. It completes the proof.
\end{proof}

\section{Asymptotic behavior of the eigenfunctions $v_{i,\epsilon}$ for $i=2,\cdots,N+1$}\label{section4}
In this section,
we first give a limit characterization of $V_{i,\epsilon}$, which is used in the proof of Theorem \ref{Figalli}.
\begin{lem}\label{identity}
For the rescaled eigenfunctions $V_{i,\epsilon}$ defined in \eqref{vie}, we have
\begin{equation}\label{viep}
V_{i,\epsilon}(x)\rightarrow\sum_{k=1}^{N}\frac{\alpha_k^ix_k}{(N(N-2)+|x|^2)^{\frac{N}{2}}}+\beta^{i}\frac{N(N-2)-|x|^2}{(N(N-2)+|x|^2)^{\frac{N}{2}}}\quad\mbox{in}\quad C_{loc}^{1}(\mathbb{R}^N)
\end{equation}
with $\alpha^{i}=(\alpha_1^{i},\cdots,\alpha_{N}^i,\beta^{i})\neq(0,\cdots,0)$ in $\mathbb{R}^{N+1}$ for $i=2,\cdots,N+1$.
\end{lem}
\begin{proof}
We find $V_{i,\epsilon}$ solves the rescaled  problem \eqref{e-1}.
Then by elliptic estimates, up to a subsequence, \eqref{U-day2} and \eqref{day2}, there exists a function $v_i$, such that
		 $$V_{i,\epsilon}\rightarrow v_i \quad\mbox{in}\quad C^1_{loc}(\mathbb{R}^n)\quad\mbox{as}\quad\epsilon\rightarrow0.$$
Moreover, by \eqref{U-day2} and \eqref{day2}, we have
\begin{equation}\label{ccinfinity}
\int_{_{\Omega_{\epsilon}}}|\nabla V_{i,\epsilon}|^2dx\leq C\int_{_{\mathbb{R}^N}}(W^{p+1}+\epsilon W^{q+1})dx,\quad\int_{_{\Omega_{\epsilon}}}|V_{i,\epsilon}|^{p+1}dx\leq C\int_{_{\mathbb{R}^N}}W^{p+1}dx.
		\end{equation}
This implies that $V_{i,\epsilon}$ is uniformly bounded in $\mathcal{D}^{1,2}(\mathbb{R}^N)$ and so $v_{i}\in\mathcal{D}^{1,2}(\mathbb{R}^N)$.
Put $\lambda_{i,\epsilon}\rightarrow\lambda_i$ as $\epsilon\rightarrow0$ and \eqref{UV}, we obtain $v_i$ satisfies
		\begin{equation}\label{Q1}
-\Delta v_{i}=\lambda_{i}\big(pW^{p-1}+\epsilon qW^{q-1}\big)v_{i} \quad \mbox{in}\quad \mathbb{R}^N,		
		\end{equation}
and $v_i\not\equiv0$. Then, using Proposition \ref{prondgr}, we get that
\begin{equation*}
v_i(x)=\sum_{k=1}^{N}\frac{\alpha_k^ix_k}{(N(N-2)+|x|^2)^{\frac{N}{2}}}+\beta^{i}\frac{N(N-2)-|x|^2}{(N(N-2)+|x|^2)^{\frac{N}{2}}}.
\end{equation*}		
In order to conclude the proof of \eqref{viep}, it is then sufficient to get that
\begin{equation}\label{vector666}
(\alpha_1^i,\alpha_2^i,\cdots,\alpha_n^i,\beta^{i})\neq0.
\end{equation}		
Let $y_{\epsilon}^{i}$ is a maximum point of $V_{i,\epsilon}$, i.e.
$V_{i,\epsilon}(y_{\epsilon}^{i})=\|V_{i,\epsilon}\|_{\infty}=1.$
 By contradiction, if \eqref{vector666} does not hold true, then necessarily $|y_{\epsilon}^{i}|\rightarrow\infty$, hence a contradiction arise with \eqref{day2}.
	\end{proof}
\begin{lem}\label{m00}
For any $i\in\{2,\dots,N+1\}$ then there exists $\beta^{i}\neq0$ such that
\begin{equation}\label{number}
\big\|u_{\epsilon}\big\|_{\infty}^2v_{i,\epsilon}(x)\rightarrow \beta^ip\frac{\omega_N}{2}\Big(\frac{\Gamma(\frac{N}{2})\Gamma(2)-\Gamma(\frac{N}{2}+1)\Gamma(1)}{\Gamma(N/2+2)}\Big)G(x,x_0)\quad\mbox{in}\quad C^{1,\alpha}(\omega),
\end{equation}
where the convergence is in $C^{1,\alpha}(\omega)$ with $\omega$ any compact set of $\bar{\Omega}$ not containing $x_0$, $x_0$ is the limit point of $x_\epsilon$.
\end{lem}
\begin{proof}
We now proceed similarly to the proof of \eqref{FINAL1}. It is noticing that
\begin{equation}\label{mn}
-\Delta\big(\big\|u_{\epsilon}\big\|^2_{\infty}v_{i,\epsilon}\big)=\lambda_{i,\epsilon}\big\|u_{\epsilon}\big\|^2_{\infty}\big(pu_{\epsilon}^{p-1}+\epsilon qu_{\epsilon}^{q-1}\big)v_{i,\epsilon}\quad\mbox{in}\quad \Omega.
\end{equation}
By dominated convergence, \eqref{viep} and Lemma \ref{limitlama}, we deduce that
\begin{equation}\label{idayu2}
\begin{split}
\lim\limits_{\epsilon\rightarrow0}\int_{\Omega}\lambda_{i,\epsilon}\big\|u_{\epsilon}\big\|^2_{\infty}\int_{\Omega}\big(pu_{\epsilon}^{p-1}+\epsilon qu_{\epsilon}^{q-1}\big)v_{i,\epsilon}&=\beta^ip\int_{\mathbb{R}^n}W^{p-1}(x)\frac{1-|x|^2}{(1+|x|^2)^{\frac{N}{2}}}
dx\\&=\beta^ip\frac{\omega_N}{2}\Big(\frac{\Gamma(\frac{N}{2})\Gamma(2)-\Gamma(\frac{N}{2}+1)\Gamma(1)}{\Gamma(N/2+2)}\Big).
\end{split}
\end{equation}
On the other hand, by \eqref{dus} and \eqref{decay22}, we have
\begin{equation*}
\lambda_{i,\epsilon}\big\|u_{\epsilon}\big\|^2_{\infty}\big(pu_{\epsilon}^{p-1}+\epsilon qu_{\epsilon}^{q-1}\big)v_{i,\epsilon}\rightarrow0\hspace{2mm}\mbox{for}\hspace{2mm} x\neq x_0.
\end{equation*}
It follows from \eqref{mn} and \eqref{idayu2} that
$$
-\Delta(\big\|u_{\epsilon}\big\|^2_{\infty}v_{i,\epsilon})\rightarrow\beta^ip\frac{\omega_N}{2}\Big(\frac{\Gamma(\frac{N}{2})\Gamma(2)-\Gamma(\frac{N}{2}+1)\Gamma(1)}{\Gamma(N/2+2)}\Big)\delta_{x=x_0}
$$
in the sense of distributions in $\Omega$ as $\epsilon\rightarrow0$.
Let $\omega$ be any compact set in $\bar{\Omega}$ not containing $x_0$. Combining Lemma \ref{regular}, we obtain
\begin{equation*}
\begin{split}
\Big\|\|u_{\epsilon}\|^2_{\infty}v_{i,\epsilon}\Big\|_{C^{1,\alpha}(\omega^{\prime})}\lesssim& \Big\|\lambda_{i,\epsilon}\big\|u_{\epsilon}\big\|^2_{\infty}\big(pu_{\epsilon}^{p-1}+\epsilon qu_{\epsilon}^{q-1}\big)v_{i,\epsilon}\Big\|_{L^1(\Omega)}\\&+\Big\|\lambda_{1,\epsilon}\big\|u_{\epsilon}\big\|^2_{\infty}\big(pu_{\epsilon}^{p-1}+\epsilon qu_{\epsilon}^{q-1}\big)v_{i,\epsilon}\Big\|_{L^\infty(\omega)}.
\end{split}
\end{equation*}
As a consequence,
$$\big\|u_{\epsilon}\big\|_{\infty}^2v_{i,\epsilon}\rightarrow\beta^ip\frac{\omega_N}{2}\Big(\frac{\Gamma(\frac{N}{2})\Gamma(2)-\Gamma(\frac{N}{2}+1)\Gamma(1)}{\Gamma(N/2+2)}\Big)G(x,x_0)\quad\mbox{in}\hspace{2mm}C^{1,\alpha}(\omega)$$
as $\epsilon\rightarrow0$ and the conclusion follows.
 \end{proof}

\section{Estimates for the eigenvalues $\lambda_{i,\epsilon}$ for $i=2,\dots,N+1$}\label{section3-1}
This section aims to establish the estimates of the eigenvalues $\lambda_{i,\epsilon}$ for $i=2,\dots,N+1$.	We first define
$\psi(x)=\tilde{\psi}(x-x_{\epsilon}),$
where $\tilde{\psi}$ is a function in $C_{0}^{\infty}(B_{2\delta}(0))$ such that $\tilde{\psi}\equiv1$ in $B_{\delta}(0)$, $0\leq\delta\leq1$ in $B_{2\delta}(x_{\epsilon})$, and
\begin{equation}\label{varepsilon}
\eta_{j,\epsilon}(x)=\psi(x)\frac{\partial u_{\epsilon}}{\partial x_j}(x),\quad j=1,\cdots,n,
\end{equation}
\begin{equation}\label{varep}
\eta_{N+1,\epsilon}(x)=\psi(x)\Big((x-x_{\epsilon})\cdot\nabla u_{\epsilon}+\frac{2}{p-1}u_{\epsilon}\Big):=\psi(x)w(x).
\end{equation}
We need some preliminary computations.
\begin{lem}\label{nabla}
It holds
\begin{equation*}
\begin{split}
-\Delta w(x)&=\big(pu_{\epsilon}^{p-1}+\epsilon qu_{\epsilon}^{q-1}\big)w(x)+\frac{2(p-q)}{p-1}\epsilon u_{\epsilon}^q
\end{split}
\end{equation*}
for any $y\in\mathbb{R}^N$.
\end{lem}
\begin{proof}
For any $y\in\mathbb{R}^N$, we get
\begin{equation}\label{comput}
\begin{split}
-\Delta w(x)&=\frac{2p}{p-1}(-\Delta u_{\epsilon})-\sum_{j=1}^{N}\sum_{i\neq j}^{N}(x_i-y_j)\frac{\partial^3u_{\epsilon}}{\partial x_i\partial^2x_j}\\&
=\frac{2p}{p-1}\left(u_{\epsilon}^{p}+\epsilon u_{\epsilon}^{q}\right)-\sum_{j=1}^{n}\sum_{i\neq j}^{n}(x_i-y_j)\frac{\partial^3u_{\varepsilon}}{\partial x_i\partial^2x_j}.
\end{split}
\end{equation}
By a straightforward computation,
\begin{equation}\label{yield}
\begin{split}
-\sum_{j=1}^{N}\sum_{i\neq j}^{N}(x_i-y_j)\frac{\partial^3u_{\varepsilon}}{\partial x_i\partial^2x_j}
=\big(pu_{\epsilon}^{p-1}+\epsilon qu_{\epsilon}^{q-1}\big)(x-y)\cdot\nabla u_{\epsilon},
\end{split}
\end{equation}
Combining \eqref{comput} and \eqref{yield} yields the conclusion.
\end{proof}	
As a consequence of Lemma \ref{nabla}, we find the following
\begin{cor}\label{conseq}
It holds that
\begin{equation*}
\begin{split}
&\int_{\partial\Omega}((x-y),\nu)\frac{\partial u_{\epsilon}}{\partial\nu}\frac{\partial v_{i,\epsilon}}{\partial\nu}dS_x=\big(1-\lambda_{i,\epsilon}\big)\int_{\Omega}\big(pu_{\epsilon}^{p-1}+\epsilon qu_{\epsilon}^{q-1}\big)wv_{i,\epsilon}dx+\frac{2(p-q)}{p-1}\epsilon \int_{\Omega}u_{\epsilon}^qv_{i,\epsilon}dx
\end{split}
\end{equation*}
for any $y\in\mathbb{R}^N$ and for $i=2,\cdots,N+2$,
where $\nu=\nu(x)$ denotes the unit outward normal to the boundary $\partial \Omega$.
\end{cor}
\begin{proof}
Recalling the identity in Lemma \ref{nabla}, and multiplying this identity by $v_{i,\epsilon}$ and integrating by parts, we deduce
\begin{equation}\label{de11}
			\begin{split}
\int_{\Omega}&\nabla w(x)\cdot\nabla v_{i,\epsilon}dx=\int_{\Omega}\big(pu_{\epsilon}^{p-1}+\epsilon qu_{\epsilon}^{q-1}\big)w(x)v_{i,\epsilon}dx+\frac{2(p-q)}{p-1}\epsilon \int_{\Omega}u_{\epsilon}^qv_{i,\epsilon}dx.
\end{split}
		\end{equation}
In same way, we also have
\begin{equation}\label{d00}
			\begin{split}
\int_{\Omega}&\nabla w(x)\cdot\nabla v_{i,\epsilon}dx-\int_{\partial\Omega}\Big((x-y)\cdot\nabla u_{\epsilon}\Big)\frac{\partial v_{i,\epsilon}}{\partial\nu}dS_x =\lambda_{i,\epsilon}\int_{\Omega}\big(pu_{\epsilon}^{p-1}+\epsilon qu_{\epsilon}^{q-1}\big)v_{i,\epsilon}w(x)
\end{split}
		\end{equation}
by \eqref{ele-2}. By plugging this identity with \eqref{de11}, this concludes the proof of Corollary \ref{conseq}.
\end{proof}
\begin{lem}\label{qiegao}
		Let $\{u_\epsilon\}_{\epsilon>0}$ be a least energy solutions of  \eqref{ele-1.1}. Assume that $v_{i,\varepsilon}$ is a solution \eqref{ele-2}. Then
		\begin{equation*}
\begin{split}
			\int_{\partial\Omega}\frac{\partial u_{\epsilon}}{\partial x_j}\frac{\partial v_{i,\epsilon}}{\partial\nu}dS_x=\big(1-\lambda_{i,\epsilon}\big)\int_{\Omega}\big(pu_{\epsilon}^{p-1}+\epsilon qu_{\epsilon}^{q-1}\big)\frac{\partial u_{\epsilon}}{\partial x_j}v_{i,\epsilon}dx
		\end{split}
\end{equation*}
		for any $j=1,\cdots,N$.
	\end{lem}
		\begin{proof}
The conclusion can be derived by simple computations similar to \eqref{de11}-\eqref{d00} and we omit it.	
	\end{proof}

We have the following results.
\begin{lem}
The functions
$u_{\epsilon},~\eta_{1,\epsilon},\cdots,\eta_{N+1,\epsilon}$
are linearly independent for $\epsilon$ sufficiently small.
\end{lem}
\begin{proof}
We now proceed similarly to \cite{GP05}. We may assume by contradiction that there exist $a_{0,\epsilon}, a_{1,\epsilon}$ $\cdots,a_{N+1,\epsilon}$ such that
\begin{equation}\label{A01}
a_{0,\epsilon}u_{\epsilon}+\sum_{k=1}^Na_{k,\epsilon}\eta_{k,\epsilon}+a_{N+1,\epsilon}\eta_{N+1,\epsilon}\equiv0\quad\mbox{and}\quad \sum_{k=1}^{N+1}a^2_{k,\epsilon}\neq0\quad\mbox{in}\quad\Omega.
\end{equation}
Without any loss of generality, we may assume that for any small, there holds
$\sum_{k=1}^{N+1}a^2_{k,\varepsilon}=1$.
We will prove that
$a_{0,\epsilon}=0$ for any $\epsilon>0$.
In fact, by contradiction, we know
$u_{\epsilon}(x)=-\sum_{k=1}^{N+1}\frac{a_{k,\epsilon}}{a_{0,\epsilon}}\eta_{k,\epsilon}.$
 It is noticing that
\begin{equation}\label{A5}
\begin{split}
-\Delta\eta_{k,\epsilon}&=\big(pu_{\epsilon}^{p-1}+\epsilon qu_{\epsilon}^{q-1}\big)\eta_{k,\epsilon}(x) \hspace{2mm}\mbox{for}\hspace{2mm}k=1,\cdots,N \hspace{2mm}\mbox{in the ball}\hspace{2mm} B_{\delta}(x_{\epsilon}).
\end{split}
\end{equation}
Owing to Lemma \ref{nabla},
\begin{equation}\label{A6}
			\begin{split}
-\Delta\eta_{N+1,\epsilon}=\big(pu_{\epsilon}^{p-1}+\epsilon qu_{\epsilon}^{q-1}\big)\eta_{N+1,\epsilon}(x)+\frac{2(p-q)}{p-1}\epsilon u_{\epsilon}^q.
\end{split}
		\end{equation}
This implies that
\begin{equation*}
\begin{split}
-\Delta\big(-\sum_{k=1}^{N+1}\frac{a_{k,\epsilon}}{a_{0,\epsilon}}\eta_{k,\epsilon}\big)=\big(pu_{\epsilon}^{p-1}+\epsilon qu_{\epsilon}^{q-1}\big)\big(-\sum_{k=1}^{N+1}\frac{a_{k,\epsilon}}{a_{0,\epsilon}}\eta_{k,\epsilon}\big)+\frac{2(p-q)}{p-1}\epsilon\big(-\frac{a_{N+1,\epsilon}}{a_{0,\epsilon}}\big) u_{\epsilon}^q.
\end{split}
\end{equation*}
Furthermore, $u_{\epsilon}$ solves \eqref{ele-1.1}, which implies
\begin{equation}\label{N+1}
\frac{a_{N+1,\epsilon}}{a_{0,\epsilon}}=\frac{(1-p)^2}{2(p-q)}u_{\epsilon}^{p-q}+(1-p)(1-q)\epsilon.
\end{equation}
On the other hand, taking $x=x_{\epsilon}$, then we have $\frac{p-1}{2}\|u_{\epsilon}\|=-\frac{a_{N+1,\epsilon}}{a_{0,\epsilon}}\|\eta_{k,\epsilon}\|$, so $\frac{a_{N+1,\epsilon}}{a_{0,\epsilon}}=-\frac{p-1}{2}<0$. This estimate gives a contradiction with \eqref{N+1}.
Inserting $x=x_{\epsilon}$ in \eqref{A01}, by virtue of $x_{\varepsilon}$ is a maximum point of $u_{\epsilon}$, we obtain
$
\frac{2a_{N+1,\epsilon}}{p-1}u_{\epsilon}(x_{\epsilon})=0,
$
and so we have $a_{N+1,\epsilon}=0$. Thus
\begin{equation}\label{afak}
\begin{split}
\sum_{k=1}^Na_{k,\epsilon}\frac{\partial U_{\epsilon}}{\partial x_j}(x)=0
\hspace{2mm}\mbox{in the ball}\hspace{2mm} B_{\delta}(x_{\epsilon}).
\end{split}
\end{equation}
Combining \eqref{UV} yields
\begin{equation*}
\sum_{k=1}^Na_{k}\frac{\partial W}{\partial x_j}(x)=0
\hspace{2mm}\mbox{in}\hspace{2mm} \mathbb{R}^N,
\end{equation*}
where $a_{k}$ is the limit of $a_{k,\epsilon}$ with $\sum_{k=1}^{N+1}a^2_{k,\epsilon}=1,$, which gives a contradiction by $\partial_{x_1}W,\cdots,\partial_{x_N}W$ are linearly independent.
\end{proof}

	\begin{lem}\label{baowenbei}
		For $i=2,\dots,N+1$, we have
		\begin{equation}\label{baowen1}
\lambda_{i,\epsilon}=1+O(\frac{M_0}{\|u_{\epsilon}\|_{\infty}^{p+1}}),
		\end{equation}
where $M_0$ is some positive dimensional constant.
\end{lem}
	\begin{proof}
We proceed similarly to \cite{GP05}. We first define a linear space $\mathcal{Z}$ spanned by $\{u_{\epsilon}\}\cup\{\eta_{j,\epsilon}: 1\leq j\leq i-1\},$
so that any nonzero function $f\in\mathcal{Z}\setminus\{0\}$ can be write as
$f=a_0u_{\epsilon}+\sum_{j=1}^{i-1}a_j\eta_{j,\epsilon}.$
By the variational
characterization of the eigenvalue $\lambda_{i,\epsilon}$, we have
		\begin{equation}\label{minmax}
\begin{split}
			\lambda_{i,\epsilon}&=\min\limits_{\substack{\mathcal{Z}\subset H_{0}^1(\Omega),\\ dim\mathcal{Z}=i}}\max\limits_{f\in\mathcal{Z}\setminus\{0\}}\frac{\displaystyle\int_{\Omega}|\nabla f(x)|^2dx}{\displaystyle\int_{\Omega}\big(pu_{\epsilon}^{p-1}+\epsilon qu_{\epsilon}^{q-1}\big)f^2dx}			\\&\leq\max\limits_{f\in\mathcal{Z}\setminus\{0\}}\frac{\displaystyle\int_{\Omega}|\nabla f(x)|^2dx}{\displaystyle\int_{\Omega}\big(pu_{\epsilon}^{p-1}+\epsilon qu_{\epsilon}^{q-1}\big)f^2dx}:=\max\limits_{f\in\mathcal{Z}\setminus\{0\}}g.
\end{split}
		\end{equation}
We next write $\omega_{\epsilon}=\sum_{j=1}^{i-1}a_j\frac{\partial u_{\epsilon}}{\partial x_j}$, so we have $f=a_0u_{\epsilon}+\psi \omega_{\epsilon}$.
Multiplying \eqref{ele-1.1} by $\psi \omega_{\epsilon}$, and
\begin{equation}\label{kuangquanshui}
			\begin{split}
				-\Delta \omega_{\epsilon}=\big(pu_{\epsilon}^{p-1}+\epsilon qu_{\epsilon}^{q-1}\big)\omega_{\epsilon}
			\end{split}
		\end{equation}
by $\psi^2\omega_{\epsilon}$ and integrating we get
		\begin{equation}\label{psi-z}
\begin{split}
			\int_{\Omega}|\nabla(\psi \omega_{\epsilon})|^2dx&=\int_{\Omega}|\nabla\psi|^2\omega^2_{\epsilon}dx+p\int_{\Omega}u_{\epsilon}^{p-1}\psi^2\omega^2_{\epsilon}dx+\epsilon q\int_{\Omega}u_{\epsilon}^{q-1}\psi^2\omega^2_{\epsilon}dx.
		\end{split}
\end{equation}
Then \eqref{kuangquanshui}-\eqref{psi-z} yields
\begin{equation*}
\begin{split}			
&\max\limits_{f\in\mathcal{Z}\setminus\{0\}}g=1+\\&\max\limits_{a_0,\cdots,a_{i-1}}\frac{a_0^2\displaystyle\int_{\Omega}\big((1-p)u_{\epsilon}^{p+1}+\epsilon (1-q)u_{\epsilon}^{q+1}\big)+2a_0\displaystyle\int_{\Omega}\big((1-p)u_{\epsilon}^{p}+\epsilon (1-q)u_{\epsilon}^{q}\big)\psi\omega_{\epsilon}}{a_0^2\displaystyle\int_{\Omega}\big(pu_{\epsilon}^{p+1}+\epsilon qu_{\epsilon}^{q+1}\big)+2a_0\displaystyle\int_{\Omega}\big(pu_{\epsilon}^{p}+\epsilon qu_{\epsilon}^{q}\big)\psi\omega_{\epsilon}+\displaystyle\int_{\Omega}\big(pu_{\epsilon}^{p-1}+\epsilon qu_{\epsilon}^{q-1}\big)\psi^2\omega^2_{\epsilon}}\\&
+\max\limits_{a_0,\cdots,a_{i-1}}\frac{\displaystyle\int_{\Omega}|\nabla\psi|^2\omega^2_{\epsilon}}{a_0^2\displaystyle\int_{\Omega}\big(pu_{\epsilon}^{p+1}+\epsilon qu_{\epsilon}^{q+1}\big)+2a_0\displaystyle\int_{\Omega}\big(pu_{\epsilon}^{p}+\epsilon qu_{\epsilon}^{q}\big)\psi\omega_{\epsilon}+\displaystyle\int_{\Omega}\big(pu_{\epsilon}^{p-1}+\epsilon qu_{\epsilon}^{q-1}\big)\psi^2\omega^2_{\epsilon}}
\end{split}
		\end{equation*}
The point is to estimate all the integrals involving $\omega_{\epsilon}$.
As far as the first one is concerned, owing
to the definitions of $\omega_{\epsilon}$ and $\psi$ one has
	\begin{equation}\label{twou1}
\begin{split}
\int_{\Omega}\big(u_{\epsilon}^{p}+\epsilon u_{\epsilon}^{q}\big)\psi\omega_{\epsilon}dx&=\sum_{j,k=1}^{i-1}a_j\int_{\Omega}u_{\epsilon}^{p-1}\psi\frac{\partial u_{\epsilon}}{\partial x_j}dx+\sum_{j,k=1}^{i-1}a_j\int_{\Omega}\epsilon u_{\epsilon}^{q-1}\psi\frac{\partial u_{\epsilon}}{\partial x_j}dx\\&
=O\Big(\frac{1}{\|u_{\epsilon}\|_{\infty}^{p+1}}\Big)+O\Big(\frac{\epsilon}{\|u_{\epsilon}\|_{\infty}^{q+1}}\Big).
\end{split}
\end{equation}
Next, from \cite{GP05}, we have
\begin{equation}\label{twou2}
\int_{\Omega}|\nabla\psi|^2\omega^2_{\epsilon}dx=\sum_{j,k=1}^{i-1}a_ja_k\int_{\Omega}|\nabla\psi|^2\frac{\partial u_{\epsilon}}{\partial x_j}\frac{\partial u_{\epsilon}}{\partial x_k}dx=O\Big(\frac{1}{\|u_{\epsilon}\|_{\infty}^{2}}\Big).
\end{equation}
Finally, we compute
		\begin{equation}\label{twou}
			\begin{split}
\int_{\Omega}\big(pu_{\epsilon}^{p-1}+\epsilon qu_{\epsilon}^{q-1}\big)&\psi^2\omega^2_{\epsilon}dx=				
\sum_{j,k=1}^{i-1}a_ja_k\int_{\Omega}\big(pu_{\epsilon}^{p-1}+\epsilon qu_{\epsilon}^{q-1}\big)\psi^2\frac{\partial u_{\epsilon}}{\partial x_j}\frac{\partial u_{\epsilon}}{\partial x_k}dx\\&
=p\sum_{j,k=1}^{i-1}a_ja_k
\big\|u_{\epsilon}\big\|_{\infty}^{p-1}\psi^2(x_0)\int_{\mathbb{R}^{N}}W^{p-1}\frac{\partial W(x)}{\partial x_j}\frac{\partial W(x)}{\partial x_k}dx\\&
+\epsilon q\sum_{j,k=1}^{i-1}a_ja_k
\big\|u_{\epsilon}\big\|_{\infty}^{q-1}\psi^2(x_0)\int_{\mathbb{R}^{N}}W^{q-1}\frac{\partial W(x)}{\partial x_j}\frac{\partial W(x)}{\partial x_k}dx+o(1)\\&
=\sum_{j,k=1}^{i-1}a_ja_k\frac{\delta_j^{k}}{N}\int_{\mathbb{R}^{N}}\big(p\big\|u_{\epsilon}\big\|_{\infty}^{p-1}W^{p-1}+\epsilon q\big\|u_{\epsilon}\big\|_{\infty}^{q-1}W^{q-1}\big)\big|\nabla W\big|^2dx+o(1).
\end{split}
\end{equation}
Combining the previous estimates together, and by taking $(a_0,\cdots,a_{i-1})=\frac{1}{i-1}(0,1,\cdots,1)$, we get
\begin{equation*}
			\begin{split}
\max\limits_{f\in\mathcal{Z}\setminus\{0\}}g&\geq1+
\frac{\sum_{j,k=1}^{i-1}a_ja_k\displaystyle\int_{\Omega}|\nabla\psi|^2\frac{\partial u_{\epsilon}}{\partial x_j}\frac{\partial u_{\epsilon}}{\partial x_k}dx}{\sum_{j,k=1}^{i-1}a_ja_k\displaystyle\int_{\Omega}\big(pu_{\epsilon}^{p-1}+\epsilon qu_{\epsilon}^{q-1}\big)\psi^2\frac{\partial u_{\epsilon}}{\partial x_j}\frac{\partial u_{\epsilon}}{\partial x_k}dx}\\&=1+\frac{1}{O\big(\|u_{\epsilon}\|_{\infty}^{p+1}\big)+O\big(\epsilon\|u_{\epsilon}\|_{\infty}^{q+1}\big)}.
\end{split}
\end{equation*}
Therefore, we conclude that
\begin{equation}\label{mo0}
\max\limits_{f\in\mathcal{Z}\setminus\{0\}}g\geq1+\frac{M_1}{\|u_{\epsilon}\|_{\infty}^{p+1}}
\end{equation}
for some constant $M_1>0$. Finally, we may assume that $\sum_{j=1}^{i-1}a_j^2=1$ and $a_{0,\epsilon},\cdots,a_{i-1,\epsilon}$ be a maximizer of $\max\limits_{a_0,\cdots,a_{i-1}}g$. By exploiting \eqref{mo0} we have $a^2_{0,\epsilon}\|u_{\epsilon}\|^2_{\infty}$ is uniformly bounded for some positive constant $C$ independent of $\epsilon$. Consequently,
\begin{equation*}
\begin{split}
&\max\limits_{f\in\mathcal{Z}\setminus\{0\}}g=1\\+&\frac{\|u_{\epsilon}\|_{\infty}^{2}a_0^2\displaystyle\int_{\Omega}\big((1-p)u_{\epsilon}^{p+1}+\epsilon (1-q)u_{\epsilon}^{q+1}\big)+O\big(\frac{1}{\|u_{\epsilon}\|_{\infty}^{p-1}}\big)+O\big(\frac{\epsilon}{\|u_{\epsilon}\|_{\infty}^{q-1}}\big)+O\big(1\big)}{\|u_{\epsilon}\|_{\infty}^{2}a_0^2\displaystyle\int_{\Omega}\big(pu_{\epsilon}^{p+1}+\epsilon qu_{\epsilon}^{q+1}\big)+O\big(\frac{1}{\|u_{\epsilon}\|_{\infty}^{p-1}}\big)+O\big(\frac{\epsilon}{\|u_{\epsilon}\|_{\infty}^{q-1}}\big)+O\big(\|u_{\epsilon}\|_{\infty}^{p+1}\big)+O\big(\epsilon\|u_{\epsilon}\|_{\infty}^{q+1}\big)}\\&
\leq1+\frac{M_0}{\|u_{\epsilon}\|_{\infty}^{p+1}}
\end{split}
\end{equation*}
for some constant $M_0>0$ and the conclusion follows.
	\end{proof}
\begin{lem}\label{limitlama}
For $i=2,\cdots,N+1$. It holds
\begin{equation}\label{baowen}
\lambda_{i,\epsilon}=1+o(1)
		\end{equation}
as $\epsilon\rightarrow0$.
	\end{lem}	
\begin{proof}
In view of \eqref{baowen1}, to conclude the proof of \eqref{baowen}, it is sufficient to prove $\lambda_{2,\epsilon}\rightarrow0$ as $\epsilon\rightarrow0$. The conclusion follows by similarly to analysis of the proof of Lemma \ref{identity}.
\end{proof}

\section{Proof of Theorem \ref{Figalli}}\label{section6}
We are now ready to prove \eqref{vie-1})in Theorem \ref{Figalli}.
\begin{proof}[Proof of Theorem \ref{Figalli}]
In order to conclude the proof of \eqref{vie-1}, it is then sufficient to get that $\beta^{i}=0$ in \eqref{identity} for any $i=\{2,\cdots,N+1\}$. To get this,
we use \eqref{ifini} and \eqref{number} to write
\begin{equation*}
\begin{split}	&\frac{\partial}{\partial\nu}\big(\big\|u_{\epsilon}\big\|_{\infty}^2u_{\epsilon}\big)\frac{\partial}{\partial\nu}\big(\big\|u_{\epsilon}\big\|_{\infty}^2v_{i,\epsilon}\big) \\&\rightarrow\beta^i\frac{p\omega_N^2}{2N}(N(N-2))^{\frac{N}{2}}\Big(\frac{\Gamma(\frac{N}{2})\Gamma(2)-\Gamma(\frac{N}{2}+1)\Gamma(1)}{\Gamma(N/2+2)}\Big)\big(\frac{\partial G}{\partial\nu}(x,x_{0})\big)^2:=R_N\beta^i\big(\frac{\partial G}{\partial\nu}(x,x_{0})\big)^2
\end{split}
\end{equation*}
uniformly on $\partial\Omega$. Furthermore, since the fact that (see \cite{BP})
\begin{equation}\label{GX0}
\int_{\partial\Omega}\frac{\partial G(x,x_0)}{\partial \nu}\frac{\partial G(x,x_0)}{\partial \nu}(\nu,x-x_0)dS_{x}=-(N-2)\phi(x_0).
\end{equation}
Then, taking $y=x_{\epsilon}$ in the left hand side of Corollary \ref{conseq}, we have
\begin{equation*}
\begin{split}
\int_{\partial\Omega}((x-x_{\epsilon}),\nu)\frac{\partial u_{\epsilon}}{\partial\nu}\frac{\partial v_{i,\epsilon}}{\partial\nu}dS_{x}
=-\frac{(N-2) R_N\beta^{i}}{\|u_{\epsilon}\|^{3}_{\infty}}\big(\phi(x_0)+o(1)\big)
\end{split}
\end{equation*}
as $\epsilon\rightarrow0$. On the other hand, by taking $y=x_{\epsilon}$ in the right hand side of Corollary \ref{conseq}, the oddness of the integrand and by the fact that
$$\big(y\cdot\nabla U_{\epsilon}+\frac{2}{p-1}U_{\epsilon}\big)\rightarrow\frac{N-2}{2}\frac{N(N-2)-|y|^2}{(N(N-2)+|y|^2)^{\frac{N}{2}}}\hspace{2mm}\mbox{in}\hspace{2mm}C_{loc}^{1}(\mathbb{R}^N),$$
we obtain
\begin{equation*}
\begin{split}
&\big(1-\lambda_{i,\epsilon}\big)\int_{\Omega}\big(pu_{\epsilon}^{p-1}+\epsilon qu_{\epsilon}^{q-1}\big)wv_{i,\epsilon}dx\big((x-x_{\epsilon})\cdot\nabla u_{\epsilon}+\frac{2}{p-1}u_{\epsilon}\big)dx\\&
=\frac{(1-\lambda_{i,\epsilon})(N+2)}{2\|u_{\epsilon}\|_{\infty}}
\beta^{i}\int_{\mathbb{R}^N}W^{p-1}(y)\Big(\frac{N(N-2)-|y|^2}{(1+|y|^2)^{\frac{N}{2}}}\Big)^2dy\\&\hspace{3mm}+\epsilon q\frac{(1-\lambda_{i,\epsilon})(N-2)}{2\|u_{\epsilon}\|_{\infty}^{p+1-q}}\beta^{i}\int_{\mathbb{R}^N}W^{q-1}(y)\Big(\frac{N(N-2)-|y|^2}{(N(N-2)+|y|^2)^{\frac{N}{2}}}\Big)^2dy+o(1)
\hspace{2mm}\mbox{as}\hspace{2mm}\epsilon\rightarrow0,
\end{split}
\end{equation*}
and
\begin{equation*}
\begin{split}
\frac{2(p-q)}{p-1}\epsilon \int_{\Omega}u_{\epsilon}^qv_{i,\epsilon}dx=\frac{2(p-q)}{p-1}\frac{\epsilon}{\|u_{\epsilon}\|_{\infty}^{p+1-q}}\beta^{i}\int_{\mathbb{R}^N}W^{q}(y)\frac{N(N-2)-|y|^2}{(N(N-2)+|y|^2)^{\frac{N}{2}}}dy+o(1).
\end{split}
\end{equation*}
Moreover, due to \eqref{Fn},
\begin{equation*}
\lim\limits_{\epsilon\rightarrow0}\frac{\epsilon}{\|u_\epsilon\|_{\infty}^{p+1-q}}=\frac{B_{p,q}\phi(x_0)}{\|u_\epsilon\|_{\infty}^3},
\end{equation*}
where $$B_{p,q}=\frac{2q}{p-q+1}\frac{(N(N-2))^{\frac{N}{2}}\omega_N}{N^2}\frac{\Gamma(\frac{(N-2)(q+1)}{2})}{\Gamma(\frac{N}{2})\Gamma(\frac{(N-2)(q+1)-N}{2})}.$$
Combining all this together and if $\beta^{i}\neq0$, we get that
\begin{equation*}
\begin{split}
\frac{1}{\|u_{\epsilon}\|^{2}_{\infty}}&\bigg[-(N-2)R_N\phi(x_0)-\frac{2(p-q)}{p-1}B_{p,q}\phi(x_0)\int_{\mathbb{R}^N}W^{q}(y)\frac{N(N-2)-|y|^2}{(N(N-2)+|y|^2)^{\frac{N}{2}}}dy+o(1)\bigg]\\&
=(1-\lambda_{i,\epsilon})\Big[\int_{\mathbb{R}^N}W^{p-1}(y)\Big(\frac{N(N-2)-|y|^2}{(N(N-2)+|y|^2)^{\frac{N}{2}}}\Big)^2dy\\&\hspace{3mm}+\frac{(N-2)B_{p,q}\phi(x_0)}{2\|u_{\epsilon}\|^{2}_{\infty}}\int_{\mathbb{R}^N}W^{q-1}(y)\Big(\frac{N(N-2)-|y|^2}{(N(N-2)+|y|^2)^{\frac{N}{2}}}\Big)^2dy\Big]
\end{split}
\end{equation*}
as $\epsilon\rightarrow0$. Next we compute two integrations
\begin{equation*}
\begin{split}
\mathcal{E}_N=:\int_{\mathbb{R}^N}W^{q}\frac{N(N-2)-|y|^2}{(N(N-2)+|y|^2)^{\frac{N}{2}}}dy&=\omega_N[N(N-2)]^{-\frac{(N-2)(q-1)}{2}}\frac{(N-2)(q-1)-N}{2}\\&\hspace{3mm}\times\Gamma(\frac{(N-2)q-2}{2})\Gamma(\frac{N}{2}),
\end{split}
\end{equation*}
and
\begin{equation*}
\begin{split}
\int_{\mathbb{R}^N}W^{q-1}&\frac{(N(N-2)-|y|^2)^2}{(N(N-2)+|y|^2)^{N}}dy=\omega_N[N(N-2)]^{-\frac{(N-2)(q-1)+N-4}{2}}\Gamma^2(\frac{N}{2})\\&\hspace{2mm}\times\frac{\Gamma(\frac{(N-2)(q-1)+N}{2})-N\Gamma(\frac{(N-2)(q-1)+N-2}{2})+\frac{N(N+2)}{4}\Gamma(\frac{(N-2)(q-1)+N-2}{2})}{2\Gamma(\frac{(N-2)(q-1)+2N}{2})}.
\end{split}
\end{equation*}
Hence we eventually deduce that
\begin{equation}\label{refor}
1-\lambda_{i,\varepsilon}=\frac{1}{\|u_{\epsilon}\|_{\infty}^{2}}\bigg[\frac{(N-2)R_N|\phi(x_0)|+\frac{2(p-q)}{p-1}B_{p,q}|\phi(x_0)|\mathcal{E}_N}{\mathcal{F}_N}+o(1)
\bigg]
\end{equation}
as $\varepsilon\rightarrow0$, where
$$\mathcal{F}_N=:\int_{\mathbb{R}^N}W^{p-1}(x)\big(\frac{1-|x|^2}{(1+|x|^2)^{\frac{N}{2}}}\big)^2dx=\frac{\pi^{\frac{N}{2}}}{2^{N+1}}>0.$$
We consider four cases, depending whether RHS of \eqref{refor} is negative or not.
\begin{itemize}
\item[Case $1$.]
If $N\geq4$ and $1<q\leq3$, then since $R_N<0$ and $\mathcal{E}_N<0$, we get
$$1-\lambda_{i,\varepsilon}=\frac{1}{\|u_{\epsilon}\|_{\infty}^{2}}\big((N-2)R_N+\frac{2(p-q)}{p-1}B_{p,q}\mathcal{E}_N+o(1)\big)<0\hspace{2mm}\mbox{as}\hspace{2mm}\epsilon\rightarrow0.$$
\item[Case $2$.]
If $N=3$ and $3<q<4$, then since $R_N<0$ and $\mathcal{E}_N<0$, we get
$$1-\lambda_{i,\varepsilon}=\frac{1}{\|u_{\epsilon}\|_{\infty}^{2}}\big((N-2)R_N+\frac{2(p-q)}{p-1}B_{p,q}\mathcal{E}_N+o(1)\big)<0\hspace{2mm}\mbox{as}\hspace{2mm}\epsilon\rightarrow0.$$
\item[Case $3$.]
If $N=3$ and $q=4$, then since $R_N<0$ and $\mathcal{E}_N=0$, we get
$$1-\lambda_{i,\varepsilon}=\frac{1}{\|u_{\epsilon}\|_{\infty}^{2}}\big((N-2)R_N+\frac{2(p-q)}{p-1}B_{p,q}\mathcal{E}_N+o(1)\big)<0\hspace{2mm}\mbox{as}\hspace{2mm}\epsilon\rightarrow0.$$
\item[Case $4$.]
If $N=3$ and $4<q<5$, we note that $\mathcal{E}_N>0$, and combing $R_N<0$, we claim that
\begin{equation*}
\begin{split}
(N-2)R_N+\frac{2(p-q)}{p-1}B_{p,q}\mathcal{E}_N=\frac{2q(5-q)(q-4)}{6-q}\frac{\Gamma(\frac{q+1}{2})\pi^3}{(\sqrt{3})^q\Gamma(3/2)}-\frac{2\pi^{3}}{\sqrt{3}\Gamma^2(3/2)}<0.
\end{split}
\end{equation*}
\end{itemize}
Indeed, we write
$$f(q):=\mathcal{Y}(q)\frac{\Gamma(\frac{q+1}{2})\pi^3}{(\sqrt{3})^q\Gamma(3/2)}-\frac{2\pi^{3}}{\sqrt{3}\Gamma^2(3/2)}\hspace{2mm}\mbox{with}\hspace{2mm}\mathcal{Y}(q):=\frac{2q(5-q)(q-4)}{6-q}.$$
It is easy to check that $\mathcal{Y}(q)\leq\mathcal{Y}(q_0)$ where $q_0\approx4.605$ for $4<q<5$, so that $f(q_0)<0$ in $q\in(4,5)$. Hence the claim holds.

In conclusion,
\begin{equation*}
1-\lambda_{i,\varepsilon}=\frac{1}{\|u_{\epsilon}\|_{\infty}^{2}}\bigg[\frac{(N-2)R_N|\phi(x_0)|+\frac{2(p-q)}{p-1}B_{p,q}|\phi(x_0)|\mathcal{E}_N}{\mathcal{F}_N}+o(1)
\bigg]<0\hspace{2mm}\mbox{as}\hspace{2mm}\epsilon\rightarrow0.
\end{equation*}
On the other hand, owing to \eqref{baowen1}, one has
$$
1-\lambda_{i,\epsilon}\geq\frac{C_N}{\|u_{\epsilon}\|_{\infty}^{\frac{2N}{N-2}}}\quad \mbox{for}\hspace{2mm}\epsilon\hspace{2mm}\mbox{sufficiently small},
$$
and some constant $C_N>0$.
Thus we clearly get a contradiction and concluding the proof.
\end{proof}

Before starting the proof of \eqref{vie-2} in Theorem \ref{Figalli}.
Then following the work of Takahashi \cite{T} for unique solvability result, we can introduce the following initial value problem of the linear first order PDE.
\begin{lem}\label{Ve}
Let some vectors $\alpha^{i}=(\alpha_1^{i},\cdots,\alpha_{N}^i)\neq0$ in $\mathbb{R}^N$ and $\Gamma_{\alpha}=\{x\in\mathbb{R}^N:\alpha\cdot z=0\}$. Then there exists a unique solution $u$ of the following initial value problem
\begin{equation}\label{kapax}
\left\lbrace
\begin{aligned}
    &\sum_{k=1}^{N}\alpha_{k}^{i}\frac{u(x)}{\partial x_k}=f,\\
    &u\big|_{\Gamma_{\alpha}}=g.
   \end{aligned}
\right.
\end{equation}
More precisely, this solution is obtained as
\begin{equation*}
u(x)=\int_{0}^{\varphi(x)}f(\tau\alpha+\gamma(\xi(x)))d\tau+g(\gamma(\xi(x))),\hspace{2mm}x\in\mathbb{R}^N
\end{equation*}
where
$$\varphi(x)=\frac{\alpha\cdot x}{|\alpha|^2},\quad\xi(x)=(\xi_1(x),\cdots,\xi_N(x)),$$
$$\xi_{j}(x)=\frac{|\alpha|^2x_j-(\alpha\cdot x)\alpha_j}{|\alpha|^2},\quad(j=1,\cdots,N-1),$$
$$\gamma(s)=(s,-\frac{1}{\alpha_N}\sum_{j=1}^{N-1}\alpha_js_j),\quad s=(s_1,\cdots,s_{N-1})\in\mathbb{R}^{N-1}$$
if we assume $\alpha_N\neq0$. Furthermore, if $f(x)=O(|x|^{\rho})$, $g(x)=O(|x|^{\rho})$ as $|x|\rightarrow\infty$, then $u(x)=O(|x|^{\rho+1})$ as $|x|\rightarrow\infty$.
\end{lem}
\begin{proof}
The unique solvability result follows from Lemma 2.4 in \cite{T}.
\end{proof}

\begin{lem}\label{6.3lemma}
Let $q$ and $N$ satisfying the condition \eqref{qN}, $\alpha^{i}=(\alpha_1^{i},\cdots,\alpha_{N}^i)\neq(0,\cdots,0)$ in \eqref{viep} and $\beta=0$ for $i=2,\cdots,N+1$. Then
\begin{equation*}
\big\|u_{\epsilon}\big\|_{\infty}^{\frac{p+3}{2}}v_{i,\epsilon}=\frac{1}{N-2}\int_{\mathbb{R}^N}W^pdy\sum_{k=1}^{N}\alpha_{k}^{i}\left(\frac{\partial}{\partial y_k}G\left(x,y\right)\right)\Big|_{y=x_{0}}
+o(1)\quad\mbox{in}\quad C_{loc}^1\big(\overline{\Omega}\setminus\{x_0\}\big).
\end{equation*}
\end{lem}
\begin{proof}
By the Green's representation, we have
\begin{equation}\label{115if}
\begin{split}
v_{i,\epsilon}(x)&=\lambda_{i,\epsilon}p\int_{\Omega}G(x,y)u_{\epsilon}^{p-1}v_{i,\epsilon}dx+\lambda_{i,\epsilon}\epsilon q\int_{\Omega}u_{\epsilon}^{q-1}v_{i,\epsilon}dx
=:I_1+I_2.
\end{split}
\end{equation}
Applying the change of variable $x=\|u_{\epsilon}\|_{\infty}^{-\frac{p-1}{2}}y+x_{\epsilon}$, we deduce that
\begin{equation}\label{and}
\begin{split}
I_1=\frac{p}{\big\|u_{\epsilon}\big\|_{\infty}^{2}}
\int_{\Omega_\epsilon}U_{\epsilon}^{p-1}(y)V_{i,\epsilon}(y) G\left(x,\frac{y}{\big\|u_{\epsilon}\big\|_{\infty}^{\frac{p-1}{2}}}+x_{\epsilon}\right)dy+o(1), \end{split}
\end{equation}
\begin{equation*}
\begin{split}
I_2=\frac{\epsilon q}{\|u_{\epsilon}\|_{\infty}^{p+2-q}}
\int_{\Omega_\epsilon}U_{\epsilon}^{q-1}(y)V_{i,\epsilon}(y) G\left(x,\frac{y}{\big\|u_{\epsilon}\big\|_{\infty}^{\frac{p-1}{2}}}+x_{\epsilon}\right)dy+o(1). \end{split}
\end{equation*}

Here we will analyze $I_1$ and $I_2$ by dividing the two cases according to the
value of $N$ and $q$.

\textbf{Case 1:} $\frac{N}{N-2}<q<p$ if $N=3$ or $N\geq4$.

Let $\mathcal{B}=(\beta_{kl})_{k,l=1,\cdots,N}$ be an orthogonal matrix which maps the vector $(|\alpha^{i}|,0,\cdots,0)$ into $\alpha=(\alpha^{i}_1,\cdots,\alpha^{i}_N)$, $i=2,\cdots,N+1$. Then for any $k=1,\cdots,N$, we have
\[
\mathcal{B}\cdot\begin{pmatrix}
|\alpha^i|\\
0\\
\vdots\\
0
\end{pmatrix}
=
\begin{pmatrix}
\alpha^i_1\\
\alpha^i_2\\
\vdots\\
\alpha^i_N
\end{pmatrix}\Longleftrightarrow \beta_{k1}=\frac{\alpha_{k}^{i}}{|\alpha^i|},
\]
and
\[
\mathcal{B}^{-1}\cdot
\begin{pmatrix}
\alpha^i_1\\
\alpha^i_2\\
\vdots\\
\alpha^i_N
\end{pmatrix}
=\begin{pmatrix}
|\alpha^i|\\
0\\
\vdots\\
0
\end{pmatrix}\Longleftrightarrow \sum_{k=1}^N\beta_{kl}\alpha_{k}^i=|\alpha^i|\delta_1^{j}.
\]
Let us consider the map $y=\mathcal{B}\xi\Longleftrightarrow y_k=\sum_{l=1}^{N}\beta_{kl}\xi_l$. Then
\begin{equation*}
\begin{split}
I_1=\frac{p}{\big\|u_{\epsilon}\big\|_{\infty}^{2}}
\int_{\mathcal{B}^{-1}(\Omega_\epsilon)}\tilde{U}_{\epsilon}^{p-1}(\xi)\tilde{V}_{i,\epsilon}(\xi) G\left(x,\frac{\mathcal{B}\xi}{\big\|u_{\epsilon}\big\|_{\infty}^{\frac{p-1}{2}}}+x_{\epsilon}\right)d\xi+o(1), \end{split}
\end{equation*}
\begin{equation*}
\begin{split}
I_2=\frac{\epsilon q}{\|u_{\epsilon}\|_{\infty}^{p+2-q}}
\int_{\mathcal{B}^{-1}(\Omega_\epsilon)}\tilde{U}_{\epsilon}^{q-1}(\xi)\tilde{V}_{i,\epsilon}(\xi) G\left(x,\frac{\mathcal{B}\xi}{\big\|u_{\epsilon}\big\|_{\infty}^{\frac{p-1}{2}}}+x_{\epsilon}\right)d\xi+o(1), \end{split}
\end{equation*}
where $\tilde{U}_{\epsilon}$ and $\tilde{V}_{\epsilon}$ are given by
\begin{equation}\label{daU}
\tilde{U}_{\epsilon}(\xi)=U_{\epsilon}\Big(\sum_{l=1}^{N}\beta_{kl}\xi_l\Big)\rightarrow W(\xi)=\Big(\frac{1}{1+|\xi|^2/N(N-2)}\Big)^{\frac{N-2}{2}},
\end{equation}
\begin{equation}\label{1daU}
\tilde{V}_{i,\epsilon}(\xi)=V_{i,\epsilon}\Big(\sum_{l=1}^{N}\beta_{kl}\xi_l\Big)\rightarrow \sum_{k=1}^{N}\frac{\alpha_k^i\sum_{l=1}^{N}\beta_{kl}\xi_l}{(N(N-2)+|\xi|^2)^{\frac{N}{2}}}=\frac{|\alpha^i|\xi_1}{(1+|\xi|^2/N(N-2))^{\frac{N}{2}}}.
\end{equation}
Introducing
\begin{equation*}
\mathcal{T}_{\epsilon}^1(\xi)=\int_{-\infty}^{\xi_1}\tilde{U}_{\epsilon}^{p-1}(s,\xi_1,\cdots,\xi_N)\tilde{V}_{i,\epsilon}(s,\xi_1,\cdots,\xi_N)ds,
\end{equation*}
and
\begin{equation*}
\mathcal{T}_{\epsilon}^2(\xi)=\int_{-\infty}^{\xi_1}\tilde{U}_{\epsilon}^{q-1}(s,\xi_1,\cdots,\xi_N)\tilde{V}_{i,\epsilon}(s,\xi_1,\cdots,\xi_N)ds,
\end{equation*}
Consequently, recalling \eqref{day2}, \eqref{U-day2} and \eqref{daU}-\eqref{1daU}, a direct calculation shows that
\begin{equation}\label{T1T2}
\mathcal{T}_{\epsilon}^1(\xi)\rightarrow-\frac{1}{N+2}\frac{|\alpha^i|}{(1+|\xi|^2)^{\frac{N}{2}}}\quad\mbox{and}\quad\mathcal{T}_{\epsilon}^2(\xi)\rightarrow-\frac{1}{(N-2)q}\frac{|\alpha^i|}{(1+|\xi|^2)^{\frac{(N-2)q}{2}}}.
\end{equation}
An integrating by parts gives
\begin{equation*}
\begin{split}
v_{i,\epsilon}(x)&=I_1+I_2\\&=
\int_{\mathcal{B}^{-1}(\Omega_\epsilon)}\Big(\frac{p}{\big\|u_{\epsilon}\big\|_{\infty}^{2}}\frac{\partial}{\partial\xi_1}(\mathcal{T}_{\epsilon}^1(\xi))+\frac{\epsilon q}{\big\|u_{\epsilon}\big\|_{\infty}^{p+2-q}}\frac{\partial}{\partial\xi_1}(\mathcal{T}_{\epsilon}^2(\xi))\Big) G\left(x,\frac{\mathcal{B}\xi}{\big\|u_{\epsilon}\big\|_{\infty}^{\frac{p-1}{2}}}+x_{\epsilon}\right)d\xi+o(1)\\&
=-\sum_{k=1}^{N}\frac{\alpha_k^i}{|\alpha^i|}\int_{\mathcal{B}^{-1}(\Omega_\epsilon)}\Big(\frac{p}{\big\|u_{\epsilon}\big\|_{\infty}^{\frac{p+3}{2}}}\mathcal{T}_{\epsilon}^1(\xi)+\frac{\epsilon q}{\big\|u_{\epsilon}\big\|_{\infty}^{\frac{3(p+1)}{2}-q}}\mathcal{T}_{\epsilon}^2(\xi)\Big) \frac{\partial G}{\partial y_k}\left(x,\frac{\mathcal{B}\xi}{\big\|u_{\epsilon}\big\|_{\infty}^{\frac{p-1}{2}}}+x_{\epsilon}\right)d\xi+o(1).
 \end{split}
\end{equation*}
To calculate the above integrals, we split its domain into $\mathcal{B}^{-1}(\Omega_\epsilon)=D_1\cup D_2$ where
$$
D_1=\left\{\xi\in\mathbb{R}^N: \frac{|\xi|}{\|u_{\epsilon}\|_{\infty}^{\frac{p-1}{2}}}<\frac{|x-x_{\epsilon}|}{2}\right\},~D_2=\left\{\xi\in\mathbb{R}^N: \frac{|\xi|}{\|u_{\epsilon}\|_{\infty}^{\frac{p-1}{2}}}\geq\frac{|x-x_{\epsilon}|}{2}\right\}.
$$
Then we divide the integral in the right-hand side of the above as
\begin{equation*}
\begin{split}
A_1+A_2+A_3+A_4:=&\Big(\int_{D_1}+\int_{D_2}\Big)\frac{p}{\big\|u_{\epsilon}\big\|_{\infty}^{(p+3)/2}}\mathcal{T}_{\epsilon}^1(\xi)\frac{\partial G}{\partial y_k}\left(x,\frac{\mathcal{B}\xi}{\big\|u_{\epsilon}\big\|_{\infty}^{\frac{p-1}{2}}}+x_{\epsilon}\right)d\xi\\&+
\Big(\int_{D_1}+\int_{D_2}\Big)\frac{\epsilon q}{\big\|u_{\epsilon}\big\|_{\infty}^{3(p+1)/2-q}}\mathcal{T}_{\epsilon}^2(\xi)\frac{\partial G}{\partial y_k}\left(x,\frac{\mathcal{B}\xi}{\big\|u_{\epsilon}\big\|_{\infty}^{\frac{p-1}{2}}}+x_{\epsilon}\right)d\xi.
\end{split}
\end{equation*}
As a consequence, for $|x-x_{\epsilon}|\geq2\delta>0$, combining \eqref{day2}, \eqref{U-day2}, \eqref{T1T2} and the following identity (see \cite{Widman})
\begin{equation}\label{GG}
\big|\nabla_{x}G(x,y)\big|\leq C|x-y|^{1-N},
\end{equation}
and passing to the limit as $\epsilon\rightarrow0$ yields
\begin{equation}\label{whence11}
\begin{split}
-\sum_{k=1}^{N}\frac{\alpha_k^i}{|\alpha^i|}A_1=\frac{1}{N-2}\frac{1}{\|u_{\epsilon}\|_{\infty}^{(p+3)/2}}\int_{\mathbb{R}^N}W^pdy\sum_{k=1}^{N}\alpha_{k}^{i}\left(\frac{G\left(x,y\right)}{\partial y_k}\right)\big|_{y=x_0}dy+o(1)\hspace{2mm}\mbox{in}\hspace{2mm}\overline{\Omega}\setminus\{x_0\}.
\end{split}
\end{equation}
Similarly, we have
\begin{equation}\label{whence11}
\begin{split}
A_3\leq C\frac{\epsilon}{\|u_{\epsilon}\|_{\infty}^{3(p+1)/2-q}}\int_{\mathbb{R}^N}\frac{1}{(1+|\xi|^2)^{\frac{(N-2)q}{2}}}d\xi\left(\frac{G\left(x,y\right)}{\partial y_k}\right)\big|_{y=x_0}dy\leq C\frac{\epsilon}{\|u_{\epsilon}\|_{\infty}^{3(p+1)/2-q}},
\end{split}
\end{equation}
for $q>\frac{N}{N-2}$ and where $3(p+1)/2-q>0$.

Next we estimate $A_2$. Using \eqref{GG}, for $|x-x_{\epsilon}|\geq2\delta>0$, we obtain
\begin{equation*}
\begin{split}
A_2&\leq C\frac{1}{\|u_{\epsilon}\|_{\infty}^{(p+3)/2}}
\int_{D_2}\frac{|\mathcal{T}_{\epsilon}^1(\xi)|}{ \big|x-x_{\epsilon}-\big\|u_{\epsilon}\big\|_{\infty}^{-\frac{p-1}{2}}\mathcal{B}\xi\big|^{N-1}}d\xi\\&
\leq C\frac{1}{\|u_{\epsilon}\|_{\infty}^{(p+3)/2}}\frac{1}{\|u_{\epsilon}\|_{\infty}^{(p-1)(N+2)/2}}
\int_{D_2}\frac{\big\|u_{\epsilon}\big\|_{\infty}^{\frac{(p-1)(N-1)}{2}}}{ \big|\big\|u_{\epsilon}\big\|_{\infty}^{\frac{p-1}{2}}(x-x_{\epsilon})-\mathcal{B}\xi\big|}d\xi\leq C\frac{1}{\|u_{\epsilon}\|_{\infty}^{(5p-1)/2}}.
 \end{split}
\end{equation*}
In same way, we have
\begin{equation*}
\begin{split}
A_4&\leq C\frac{\epsilon}{\|u_{\epsilon}\|_{\infty}^{3(p+1)/2-q}}
\int_{D_2}\frac{|\mathcal{T}_{\epsilon}^2(\xi)|}{ \big|x-x_{\epsilon}-\big\|u_{\epsilon}\big\|_{\infty}^{-\frac{p-1}{2}}\mathcal{B}\xi\big|^{N-1}}d\xi\\&
\leq C\frac{\epsilon}{\|u_{\epsilon}\|_{\infty}^{3(p+1)/2-q}}\frac{1}{\|u_{\epsilon}\|_{\infty}^{(p-1)(N-2)q/2}}
\int_{D_2}\frac{\big\|u_{\epsilon}\big\|_{\infty}^{\frac{(p-1)(N-1)}{2}}}{ \big|\big\|u_{\epsilon}\big\|_{\infty}^{\frac{p-1}{2}}(x-x_{\epsilon})-\mathcal{B}\xi\big|}d\xi\\&
\leq C\frac{1}{\big\|u_{\epsilon}\big\|_{\infty}^{3(p+1)/2-q+(p-1)/2[(N-2)q+2-N]}},
 \end{split}
\end{equation*}
where $3(p+1)/2-q+(p-1)/2[(N-2)q+2-N]>0$ since $q>1$.

In conclusion, the above all estimates tell us that
\begin{equation}\label{vix}
\big\|u_{\epsilon}\big\|_{\infty}^{\frac{p+3}{2}}v_{i,\epsilon}(x)=\frac{1}{N-2}\int_{\mathbb{R}^N}W^pdy\sum_{k=1}^{N}\alpha_{k}^{i}\left(\frac{G\left(x,y\right)}{\partial y_k}\right)\big|_{y=x_0}dy+o(1)\hspace{2mm}\mbox{in}\hspace{2mm}\overline{\Omega}\setminus\{x_0\}.
\end{equation}

\textbf{Case 2:} $\frac{3}{N-2}<q\leq\frac{N}{N-2}$ if $N\geq4$.

By \eqref{UV} and \eqref{viep}, we obtain
$$U_{\epsilon}^{p-1}V_{i,\epsilon}\rightarrow\sum_{k=1}^{N}\alpha_{k}^{i}\frac{\partial}{\partial y_k}\big(-\frac{W^p(y)}{N+2}\big).$$
Then in view of Lemma \ref{Ve}, we have the linear first order PDE
\begin{equation*}
\left\lbrace
\begin{aligned}
    &\sum_{k=1}^{N}\alpha_{k}^{i}\frac{w(y)}{\partial y_k}=U_{\epsilon}^{p-1}V_{i,\epsilon},\\
    &w\big|_{\Gamma_{\alpha}}=-\frac{W^p(y)}{N+2},
   \end{aligned}
\right.
\end{equation*}
and whence there exist a solution the above problem $w_{\epsilon}$ such that
$$\int_{\mathbb{R}^N}w(y)dy\rightarrow-\int_{\mathbb{R}^N}\frac{W^p(y)}{N+2}dy.$$
As a consequence, an integration by parts in \eqref{and} and passing to the limit as $\epsilon\rightarrow0$ yields
\begin{equation}\label{whence11}
\begin{split}
J_1&=\frac{p}{\|u_{\epsilon}\|_{\infty}^{2}}
\int_{\Omega_\epsilon}\sum_{k=1}^{N}\alpha_{k}^{i}\frac{w(y)}{\partial y_k} G\left(x,\frac{y}{\big\|u_{\epsilon}\big\|_{\infty}^{\frac{p-1}{2}}}+x_{\epsilon}\right)dy+o(1)\\&
=\frac{1}{N-2}\frac{1}{\|u_{\epsilon}\|_{\infty}^{(p+3)/2}}\int_{\mathbb{R}^N}W^pdy\sum_{k=1}^{N}\alpha_{k}^{i}\left(\frac{G\left(x,y\right)}{\partial y_k}\right)\big|_{y=x_0}dy+o(1)\hspace{2mm}\mbox{in}\hspace{2mm}\overline{\Omega}\setminus\{x_0\}.
\end{split}
\end{equation}
Next we estimate $J_2$. Recalling $J_2$ and
we consider the following initial value problem
\begin{equation*}
\sum_{k=1}^{N}\alpha_{k}^{i}\frac{w(y)}{\partial y_k}=U_{\epsilon}^{q-1}V_{i,\epsilon},\quad w\big|_{\Gamma_{\alpha}}=-\frac{1}{q(N-2)}W^q(y).
\end{equation*}
Recalling Lemma \ref{Ve}, there exist a solution $w_{\epsilon}$ of the initial value problem satisfying the estimate
$$w_{\epsilon}=O(|y|^{-(N-2)q-1}),$$
because of $U_{\epsilon}^{q-1}V_{i,\epsilon}=O(|y|^{-(N-2)q})$ and $-\frac{1}{q(N-2)}W^q(y)=O(|y|^{-(N-2)q})$.
Therefore, note that $\Omega_{\epsilon}\subset B_{c\|u_{\epsilon}\|_{\infty}^{(p-1)/2}}(0)$ for some $c>0$, we have
\begin{equation*}
\begin{split}
\Big|\int_{\Omega_{\epsilon}}w_{\epsilon}(y)dy\Big|&\leq C\int_{B_{1}(0)}\frac{1}{|y|^{(N-2)q-1}}dy+C\int_{B_{c\|u_{\epsilon}\|_{\infty}^{(p-1)/2}}(0)\setminus {B_{1}(0)}}\frac{1}{|y|^{(N-2)q-1}}dy\\&
\leq C+C\int^{c\|u_{\epsilon}\|_{\infty}^{(p-1)/2}}_{1}r^{N-(N-2)q}dr\\&
\leq C\big\|u_{\epsilon}\big\|_{\infty}^{\frac{p-1}{2}(N+1-(N-2)q)},
\end{split}
\end{equation*}
where $1<q\leq\frac{N}{N-2}$.
Then similarly to the argument of $J_1$, it holds that
\begin{equation*}
\begin{split}
\big\|u_{\epsilon}\big\|_{\infty}^{\frac{p+3}{2}}|J_2|&\leq\frac{\epsilon q}{\|u_{\epsilon}\|_{\infty}^{p+2-q-(p+3)/2}}
\int_{\Omega_\epsilon}\sum_{k=1}^{N}\alpha_{k}^{i}\frac{w(y)}{\partial y_k} G\left(x,\frac{y}{\big\|u_{\epsilon}\big\|_{\infty}^{\frac{p-1}{2}}}+x_{\epsilon}\right)dy\\&
\leq C\frac{\epsilon}{\|u_{\epsilon}\|_{\infty}^{p+2-q+\frac{p-1}{2}}}\big\|u_{\epsilon}\big\|_{\infty}^{\frac{p+3}{2}}\int_{\Omega_\epsilon}\sum_{k=1}^{N}\alpha_{k}^{i}\left(\frac{G\left(x,z\right)}{\partial z_k}\right)\Big|_{z=(\|u_{\epsilon}\|_{\infty}^{-\frac{p-1}{2}}y+x_{\epsilon})}w_{\epsilon}(y)dy\\&
\leq C\frac{1}{\|u_{\epsilon}\|_{\infty}^{2}}\big|\int_{\Omega_\epsilon} w_{\epsilon}(y)dy\big|\leq C\frac{1}{\|u_{\epsilon}\|_{\infty}^{2(q-\frac{3}{N-2})}}=o(1),
\end{split}
\end{equation*}
as $\epsilon$ sufficiently small, where we used \eqref{2-miu} and $q>\frac{3}{N-2}$.
Summarizing, by plugging this last estimate with \eqref{whence11} in \eqref{115if}, we get that
\begin{equation}\label{vix12}
\big\|u_{\epsilon}\big\|_{\infty}^{\frac{p+3}{2}}v_{i,\epsilon}=\frac{1}{N-2}\int_{\mathbb{R}^N}W^pdy\sum_{k=1}^{N}\alpha_{k}^{i}\left(\frac{G\left(x,y\right)}{\partial y_k}\right)\big|_{y=x_0}dy+o(1)\hspace{2mm}\mbox{in}\hspace{2mm}\overline{\Omega}\setminus\{x_0\}.
\end{equation}

Finally, the conclusion follows by \eqref{vix} and \eqref{vix12}, which completes the proof.
\end{proof}
We shall give the proof of \eqref{vie-2}.
\begin{proof}[Proof of Theorem \ref{Figalli}]
The conclusion follows by Lemma \ref{6.3lemma}.
\end{proof}

	\section{Proof of Theorem \ref{remainder terms}}\label{sangshen}
The objective of this section is to give the proof of Theorem \ref{remainder terms}. The following Lemmas will be used to obtain the asymptotic behaviour of the eigenvalues $\lambda_{i,\epsilon}$ for $i=2,\cdots,N+1$.	
	\begin{lem}\label{wanfan}
		Let $v_{k,\epsilon}$ and  $v_{l,\epsilon}$ are two orthogonal eigenfunctions of \eqref{ele-2} in $H_{0}^{1}(\Omega)$ corresponding to eigenvalues $\lambda_{k,\epsilon}$ and $\lambda_{l,\epsilon}$ , then there holds	$$
(\alpha^{k},\alpha^{l})=0,
		$$
where the corresponding vector $\alpha^{l}$ and $\alpha^{m}$ defined in Theorem \ref{Figalli}.
	\end{lem}
	\begin{proof}
		By in virtue of \eqref{ele-2} and orthogonality, we obtain
		\begin{equation}\label{pingjiehuoguo}
		\begin{split}	
\int_{\Omega}\big(pu_{\epsilon}^{p-1}+\epsilon qu_{\epsilon}^{q-1}\big)v_{k,\epsilon}(x)v_{l,\epsilon}(x)dx=0.
\end{split}
		\end{equation}
The desired result can be achieved by rescaling the functions $u_{\epsilon}$, $v_{k,\epsilon}$ and $v_{l,\epsilon}$ in \eqref{pingjiehuoguo}.
\end{proof}
	
We can now prove Theorem \ref{remainder terms}.
\begin{proof}[Proof of Theorem \ref{remainder terms}]
Let $L_{\epsilon}$ and $R_{\epsilon}$ be the left-hand and right-hand side of the identity in Lemma \ref{qiegao}, respectively, so that $L_{\epsilon}=R_{\epsilon}$. It follows directly from \eqref{vie-2}, \eqref{ifini} and
\begin{equation}\label{GX0}
\int_{\partial\Omega}\frac{\partial G(x,w)}{\partial x_j}\frac{\partial G(x,w)}{\partial w_k}\Big(\frac{\partial G}{\partial\nu}(x,w)\Big)dS_{x}=\frac{1}{2}\frac{\partial^2}{\partial x_j\partial x_k}\phi(w)\quad\mbox{for all}\hspace{2mm}j, k=1,\cdots,N,
\end{equation}
that
\begin{equation}\label{caI}
\begin{split}		L_{\epsilon}&=\big\|u_{\epsilon}\big\|_{\infty}^{-\frac{p+5}{2}}\int_{\partial\Omega}\frac{\partial }{\partial x_j}\big(\big\|u_{\epsilon}\big\|_{\infty}u_{i,\epsilon}\big)\frac{\partial }{\partial\nu}\big(\big\|u_{\epsilon}\big\|_{\infty}^{\frac{p+3}{2}}v_{i,\epsilon}\big)dS_x\\&
=\big\|u_{\epsilon}\big\|_{L^{\infty}(\Omega)}^{-\frac{p+5}{2}}\Big[(N(N-2))^{\frac{N-2}{2}}a_{N}\omega_N^2\sum_{k=1}^{n}\alpha_k^i\frac{\partial^2}{\partial x_j\partial x_k}\phi(x_0)+o(1)\Big]
\end{split}
\end{equation}
Next we estimate $R_{\epsilon}$, i.e.,
		\begin{equation}\label{AIC}
\begin{split}
R_{\epsilon}=\big(1-\lambda_{i,\epsilon}\big)\int_{\Omega}pu_{\epsilon}^{p-1}\frac{\partial u_{\epsilon}}{\partial x_j}v_{i,\epsilon}dx+\big(1-\lambda_{i,\epsilon}\big)\epsilon q\int_{\Omega}u_{\epsilon}^{q-1}\frac{\partial u_{\epsilon}}{\partial x_j}v_{i,\epsilon}dx.
\end{split}
\end{equation}
A direct calculation tells us that
\begin{equation*}
\begin{split} \big(1-\lambda_{i,\epsilon}\big)\int_{\Omega}pu_{\epsilon}^{p-1}\frac{\partial u_{\epsilon}}{\partial x_j}v_{i,\epsilon}dx
=\frac{\lambda_{i,\epsilon}-1}{\big\|u_{\epsilon}\big\|_{\infty}^{2-\frac{p+1}{2}}}\frac{p}{N(N-2)}\Big[\alpha_j^i\int_{\mathbb{R}^N}W^{p-1}|\nabla W|^2dx+o(1)\Big]
\end{split}
\end{equation*}
by \eqref{viep}, the usual change of variables $x=\|u_{\epsilon}\|_{\infty}^{-\frac{p-1}{2}}y+x_{\epsilon}$ and the dominated convergence. Analogously, we have
\begin{equation*}
\begin{split}
\big(1-\lambda_{i,\epsilon}\big)\epsilon q\int_{\Omega}u_{\epsilon}^{q-1}\frac{\partial u_{\epsilon}}{\partial x_j}v_{i,\epsilon}dx&=\frac{\lambda_{i,\epsilon}-1}{\|u_{\epsilon}\|_{\infty}^{(p+3)/2-q}}\frac{\epsilon q}{N-2}\Big[\alpha_j^i\int_{\mathbb{R}^N}W^{q-1}(\frac{\partial W}{\partial x_j})^2dx+o(1)\Big]\\&
=(\lambda_{i,\epsilon}-1)O\Big(\big\|u_{\epsilon}\big\|_{\infty}^{\frac{p-7}{2}}\Big)+o(1)
\end{split}
\end{equation*}
by \eqref{2-miu}. Combining the estimates for $L_{\epsilon}$ and $R_{\epsilon}$ together, we obtain
\begin{equation}\label{FR}
\begin{split}
&(N(N-2))^{\frac{N-2}{2}}a_{N}\omega_N^2\sum_{k=1}^{n}\alpha_k^i\frac{\partial^2}{\partial x_j\partial x_k}\phi(x_0)+o(1)\\=&(\lambda_{i,\epsilon}-1)\Big[\alpha_j^i\big\|u_{\epsilon}\big\|_{\infty}^{p+1}\frac{p}{N(N-2)}\int_{\mathbb{R}^N}W^{p-1}|\nabla W|^2dx+O\Big(\big\|u_{\epsilon}\big\|_{\infty}^{p-1}\Big)\Big].
\end{split}
\end{equation}
Owing to $\alpha_k^i\neq0$ in Theorem \ref{Figalli},
\begin{equation}\label{FR-1}
(\lambda_{i,\epsilon}-1)\big\|u_{\epsilon}\big\|_{\infty}^{p+1}\rightarrow \mathcal{M}\frac{\sum_{k=1}^{n}\alpha_k^i\frac{\partial^2}{\partial x_j\partial x_k}\phi(x_0)}{\alpha_j^i}=:\mathcal{M}\zeta_i,
\end{equation}
where
$$\mathcal{M}:=\frac{(N(N-2))^{\frac{N}{2}}a_{N}\omega_N^2}{p\int_{\mathbb{R}^N}W^{p-1}|\nabla W|^2dx}.$$	
Thanks to \eqref{FR} and \eqref{FR-1}, there holds
$\sum_{k=1}^{n}\alpha_k^i\frac{\partial^2}{\partial x_j\partial x_k}\phi(x_0)=\zeta_i\alpha_j^i.$
This means that $\zeta_i$ is an eigenvalue of the Hessian matrix $D^{2}\phi(x_0)$ with $\alpha^i$ as corresponding eigenvector. By Lemma \ref{wanfan}, we have that $\zeta_i$ for $i=2,\cdots,N+1$ are all the $N$ eigenvalues $\nu_1,\cdots,\nu_N$ of the Hessian matrix $D^{2}\phi(x_0)$. Moreover, combining  $\lambda_{2,\epsilon}\leq\cdots\leq\lambda_{N+1,\epsilon}$ and \eqref{FR-1}, we get $\zeta_i=\nu_{i-1}$ for $i=2,\cdots,N+1$.
Concluding the proof of Theorem \ref{remainder terms}.
	\end{proof}

\section{Proof of Theorems \ref{eigenfunctions-1} and \ref{eigenfunctio1}}\label{section9}	
As a preparation step for the proof of Theorem \ref{eigenfunctions-1}, we need a convergence result which will be employed in the proof of $(b)$ in Theorem \ref{eigenfunctions-1}.
\begin{lem}\label{1--identity}
Let $\{y_N\}\in C^1(\mathbb{R}^N)$ be a sequence of functions such that
\begin{equation*}
y_N\rightarrow\sum_{k=1}^{N}\frac{\gamma_kx_k}{(N(N-2)+|x|^2)^{\frac{N}{2}}}\quad\mbox{in}\quad C_{loc}^{1}(\mathbb{R}^N)
\end{equation*}
with $\gamma=(\gamma_1,\cdots,\gamma_N)\neq(0,\cdots,0)\in\mathbb{R}^N$. Assume that $\mathscr{Y}_+:=\{x\in\mathbb{R}^N:~y_N>0\}$ and $\mathscr{Y}_-:=\{x\in\mathbb{R}^N:~y_N<0\}$.
Then for any $R>0$, the sets $\mathscr{Y}_+\cap B_R$ and $\mathscr{Y}_-\cap B_R$ are both connected and nonempty for sufficiently large and $B_R=\{x\in\mathbb{R}^N:~|x|<R\}$.
\end{lem}
\begin{proof}
The proof is essentially the same as that of \cite[Lemma 6.1]{GP05}, which we omit.
\end{proof}
\begin{proof}[Proof of Theorem \ref{eigenfunctions-1}]
The proof of part $(a)$ is the same as the one given in \cite{GP05}. By contradiction, part $(b)$ follows from Lemma \ref{1--identity}.
\end{proof}

\begin{proof}[Proof of Theorem \ref{eigenfunctio1}]
This follows from the same argument as in \cite{GP05} but we will just sketch it for reader's convenience.

First of all, we ruled out the full even symmetry of $v_{i,\epsilon}$ by contradiction.

Fix  $i \in \{2, \dots, N+1\}$ , and assume that for sufficiently small  $\epsilon$, the eigenfunction  $v_{i,\epsilon}$ is evenly symmetric with respect to all variables  $x_j$ for $(j=1, \dots, N)$.
In this case, by performing rescaling and taking the limit of equation \eqref{vie-1} in Theorem \ref{Figalli}, it can be deduced that the function $\sum_{k=1}^N \frac{\alpha_k^i x_k}{(N(N-2) + |x|^2)^{N/2}}$ is also evenly symmetric with respect to all variables. However, this is impossible.

Next, by Gidas-Ni-Nirenberg theorem \cite{GNN}, $u_{\epsilon}$ is symmetric with respect to the hyperplanes $\mathcal{T}_{j}=\{x\in\mathbb{R}^N,~x_j=0\}$ for $j=1,\cdots,N$. Combining the symmetry of $\Omega$, each eigenfunction  $v_{i,\epsilon}$ can be decomposed into an even part plus an odd part. Based on the properties of this decomposition, we can conclude that each eigenfunction  $v_{i,\epsilon}$ is either oddly symmetric or evenly symmetric with respect to each variable $x_j~(j=1,\dots, N)$.
Therefore, there must exist at least one variable $x_j$ such that $v_{i,\epsilon}$ is oddly symmetric with respect to this variable.
Furthermore, owing $(b)$ in Theorem \ref{eigenfunctions-1} above, it can be known that  $v_{i,\epsilon}$ must be the first eigenfunction (20) in $\mathscr{D}_{j}^{-}$, hence $\lambda_{i,\epsilon}=\mu_{j,\epsilon}^{(1)}~(i=2,\cdots,N+1)$ for sufficiently small  $\epsilon$. This completes the proof of Theorem \ref{eigenfunctio1}.
\end{proof}

\section{Proof of Theorems \ref{thmprtb} and \ref{thmprtb-1}}\label{section8}
In this section, we are devoted to show that Theorem \ref{thmprtb} and Corollary \ref{emm-1}. 	
We first need estimate for $\lambda_{N+2,\epsilon}$, which is similar to the treatment in Lemmas \ref{baowenbei} and \ref{limitlama}.
\begin{lem}\label{ba}
It holds
\begin{equation}
\lambda_{N+2,\epsilon}=1+o(1)
		\end{equation}
as $\epsilon\rightarrow0$.
\end{lem}
\begin{proof}
By virtue of Theorem \ref{remainder terms} and Lemma \ref{limitlama},
it is sufficient to prove $$\limsup_{\epsilon\rightarrow0}\lambda_{N+2,\epsilon}\leq1.$$
Referring to \eqref{varep}, we let $\tilde{\mathcal{Z}}$ be a vector space whose basis is $\{u_{\epsilon}\}\cup\{\eta_{j,\epsilon}: 1\leq j\leq N+1\},$
so that any nonzero function $\tilde{f}\in\tilde{\mathcal{Z}}\setminus\{0\}$ can be write as
\begin{equation}\label{nonzero-00}
\tilde{f}=a_0u_{\epsilon}+\psi\left(\sum_{j=1}^{N}a_j\frac{\partial u_{\epsilon}}{\partial x_j}+dw(x)\right):=a_0u_{\epsilon}+\psi\tilde{\omega}_{\epsilon},
\end{equation}
for some $a_0,~a_j,~d\in\mathbb{R}$.
Then we have
		\begin{equation}\label{minmax-00}
\begin{split}
			\lambda_{N+2,\epsilon}&=
\leq\max\limits_{\tilde{f}\in\tilde{\mathcal{Z}}\setminus\{0\}}\frac{\displaystyle\int_{\Omega}|\nabla \tilde{f}(x)|^2dx}{\displaystyle\int_{\Omega}\big(pu_{\epsilon}^{p-1}+\epsilon qu_{\epsilon}^{q-1}\big)f^2dx}
:=\max\limits_{\tilde{f}\in\tilde{\mathcal{Z}}\setminus\{0\}}(1+\frac{l}{h}),
\end{split}
		\end{equation}
where $h:=h_1+h_2+h_3$ and $l:=l_1+l_2+l_3$ with
\begin{equation*}
\begin{split}
&h_1=a_0^2\displaystyle\int_{\Omega}\big(pu_{\epsilon}^{p+1}+\epsilon qu_{\epsilon}^{q+1}\big),\\&
h_2=2a_0\displaystyle\int_{\Omega}\big(pu_{\epsilon}^{p}+\epsilon qu_{\epsilon}^{q}\big)\psi\left(\sum_{j=1}^{N}a_j\frac{\partial u_{\epsilon}}{\partial x_j}+dw(x)\right),\\&
h_3=\int_{\Omega}\big(pu_{\epsilon}^{p-1}+\epsilon qu_{\epsilon}^{q-1}\big)\psi^2\left(\sum_{j,k=1}^{N}a_ja_k\frac{\partial u_{\epsilon}}{\partial x_j}\frac{\partial u_{\epsilon}}{\partial x_k}+2d\sum_{j=1}^{N}a_j\frac{\partial u_{\epsilon}}{\partial x_j}w(x)+d^2w^2(x)\right),
\\&l_1=a_0^2\int_{\Omega}\big((1-p)u_{\epsilon}^{p+1}+\epsilon (1-q)u_{\epsilon}^{q+1}\big),\\&
l_2=2a_0\int_{\Omega}\big((1-p)u_{\epsilon}^{p}+\epsilon (1-q)u_{\epsilon}^{q}\big)\psi\left(\sum_{j=1}^{N}a_j\frac{\partial u_{\epsilon}}{\partial x_j}+dw(x)\right),\\&
l_3=\int_{\Omega}|\nabla\psi|^2\left(\sum_{j,k=1}^{N}a_ja_k\frac{\partial u_{\epsilon}}{\partial x_j}\frac{\partial u_{\epsilon}}{\partial x_k}+2d\sum_{j=1}^{N}a_j\frac{\partial u_{\epsilon}}{\partial x_j}w(x)+d^2w^2(x)\right).
\end{split}
\end{equation*}
Similar to the proof of in \cite{GP05}. It follows from \eqref{2-miu} and the proof of Lemma \ref{baowenbei} that
\begin{equation*}
h_2=O\Big(\frac{1}{\|u_{\epsilon}\|_{\infty}^{p+1}}\Big)+O\Big(\frac{\epsilon}{\|u_{\epsilon}\|_{\infty}^{q+1}}\Big)=O\Big(\frac{1}{\|u_{\epsilon}\|_{\infty}^{2}}\Big),
\end{equation*}
\begin{equation*}
l_2=O\Big(\frac{1}{\|u_{\epsilon}\|_{\infty}^{p+1}}\Big)+O\Big(\frac{\epsilon}{\|u_{\epsilon}\|_{\infty}^{q+1}}\Big)=O\Big(\frac{1}{\|u_{\epsilon}\|_{\infty}^{2}}\Big).
\end{equation*}
By \eqref{twou1} and \eqref{twou}, we have
\begin{equation*}
\begin{split}
h_3&=\sum_{j,k=1}^{N}a_ja_k\frac{\delta_j^{k}}{N}\int_{\mathbb{R}^{N}}\big(p\big\|u_{\epsilon}\big\|_{\infty}^{p-1}W^{p-1}+\epsilon q\big\|u_{\epsilon}\big\|_{\infty}^{q-1}W^{q-1}\big)\big|\nabla W\big|^2dx\\&
+\sum_{j=1}^{N}a_j\Big[\frac{2dp}{\|u_{\epsilon}\|_{\infty}^{\frac{(p-1)(N+1)}{2}-2p}}\int_{\Omega_{\epsilon}}U_{\epsilon}^{p-1}y_k\frac{\partial U_{\epsilon}}{\partial x_j}\frac{\partial U_{\epsilon}}{\partial x_k}+\frac{2d\epsilon q}{\|u_{\epsilon}\|_{\infty}^{\frac{(p-1)(N+1)}{2}-p-q}}\int_{\Omega_{\epsilon}}U_{\epsilon}^{q-1}y_k\frac{\partial U_{\epsilon}}{\partial x_j}\frac{\partial U_{\epsilon}}{\partial x_k}\Big]
\\&+pd^2\int_{\mathbb{R}^N}W^{p-1}(y)\Big(\frac{N(N-2)-|y|^2}{(1+|y|^2)^{\frac{N}{2}}}\Big)^2dy
+\epsilon qd^2\int_{\mathbb{R}^N}W^{p-1}(y)\Big(\frac{N(N-2)-|y|^2}{(1+|y|^2)^{\frac{N}{2}}}\Big)^2dy
\\&+O(\|u_{\epsilon}\|_{\infty}^{-(p+1)})+O(\|u_{\epsilon}\|_{\infty}^{-(q+1)})+o(1)\\&
=\sum_{j,k=1}^{N}a_ja_k\frac{\delta_j^{k}}{N}\int_{\mathbb{R}^{N}}\big(p\big\|u_{\epsilon}\big\|_{\infty}^{p-1}W^{p-1}+\epsilon q\big\|u_{\epsilon}\big\|_{\infty}^{q-1}W^{q-1}\big)\big|\nabla W\big|^2dx\\&
+pd^2\int_{\mathbb{R}^N}W^{p-1}(y)\Big(\frac{N(N-2)-|y|^2}{(1+|y|^2)^{\frac{N}{2}}}\Big)^2dy
+\epsilon qd^2\int_{\mathbb{R}^N}W^{p-1}(y)\Big(\frac{N(N-2)-|y|^2}{(1+|y|^2)^{\frac{N}{2}}}\Big)^2dy\\&
+2dp\sum_{j=1}^{N}a_jo(\|u_{\epsilon}\|_{\infty}^{\frac{2}{N-2}})+O\Big(\frac{\epsilon }{\|u_{\epsilon}\|_{\infty}^{\frac{(p-1)(N+1)}{2}-p-q}}\Big)+O(\|u_{\epsilon}\|_{\infty}^{-(q+1)})+o(1).
\end{split}
\end{equation*}
Finally, from the proof of  Lemmas \ref{baowenbei}, we have
$$l_3=O(\frac{1}{\|u_{\epsilon}\|_{\infty}^{2}}).$$
Therefore, via the estimates of $h$ and $l$ together, we conclude that
\begin{equation}\label{617}
\begin{split}
h&\geq C_1(|a|\big\|u_{\epsilon}\big\|_{\infty}^{p-1}+\epsilon q\big\|u_{\epsilon}\big\|_{\infty}^{q-1})+C_2d^2p+C_3d^2\epsilon q\geq\delta>0,
\end{split}
\end{equation}
for some positive constant $\delta$ and
\begin{equation}\label{618}
\begin{split}
l=a_0^2\int_{\Omega}\big((1-p)u_{\epsilon}^{p+1}+\epsilon (1-q)u_{\epsilon}^{q+1}\big)+O(\frac{1}{\|u_{\epsilon}\|_{\infty}^{2}})\leq C\frac{1}{\|u_{\epsilon}\|_{\infty}^{2}}.
\geq\delta>0
\end{split}
\end{equation}
It follows from \eqref{minmax-00} and \eqref{617}-\eqref{618} that
\begin{equation*}
\limsup_{\epsilon\rightarrow0}\lambda_{N+2,\epsilon}\leq \limsup_{\epsilon\rightarrow0}\max\limits_{\tilde{f}\in\tilde{\mathcal{Z}}\setminus\{0\}}\big(1+\frac{l}{h}\big)\leq1,
\end{equation*}
and the conclusion follows.
\end{proof}

 Now we are ready to prove Theorem \ref{thmprtb}.
\begin{proof}[Proof of Theorem \ref{thmprtb}]	
We now proceed similarly to Lemma \ref{identity}. For the rescaled eigenfunctions $\tilde{v}_{n+2,\epsilon}$, we obtain
\begin{equation}\label{viep-1}
V_{N+2,\epsilon}(x)\rightarrow\sum_{k=1}^{N}\frac{\alpha_kx_k}{(N(N-2)+|x|^2)^{\frac{N}{2}}}+\beta\frac{N(N-2)-|x|^2}{(N(N-2)+|x|^2)^{\frac{N}{2}}}\quad\mbox{in}\hspace{2mm} C_{loc}^{1}(\mathbb{R}^n)
\end{equation}
with $(\alpha_1,\cdots,\alpha_{N},\beta)\neq(0,\cdots,0)$ in $\mathbb{R}^{N+1}$. We claim that $(\alpha_1,\cdots,\alpha_{N})=(0,\cdots,0)$. Observe that once this claim is established, \eqref{rtbdy46} will follow. Indeed,  arguing in the same way to Lemma \ref{wanfan}, we know if the eigenfunctions $v_{N+2,\epsilon}$ and $v_{j,\epsilon}$ are orthogonal in $H_{0}^{1}(\Omega)$ for $j=2,\cdots,N+1$. Then we have
\begin{equation}\label{pingjiehuoguo}
		\begin{split}	
\int_{\Omega}\big(pu_{\epsilon}^{p-1}+\epsilon qu_{\epsilon}^{q-1}\big)v_{N+2,\epsilon}(x)v_{j,\epsilon}(x)dx=0,
\end{split}
		\end{equation}
which implies $(a,a^{(m)})=0$ for $m=2,\cdots, N+1$ by rescaling the functions $u_{\epsilon}$, $v_{k,\epsilon}$ and $v_{l,\epsilon}$ in \eqref{pingjiehuoguo} and \eqref{viep-1}. Hence the claim follows from Lemma \ref{wanfan}.

Next we prove \eqref{rt}. Combining \eqref{dus}, \eqref{decay22}, Lemma \ref{regular} and Lemma \ref{ba}, then we get
\begin{equation}\label{number-1}
\big\|u_{\epsilon}\big\|_{\infty}^2v_{n+2,\epsilon}(x)\rightarrow \beta p\frac{\omega_N}{2}\Big(\frac{\Gamma(\frac{N}{2})\Gamma(2)-\Gamma(\frac{N}{2}+1)\Gamma(1)}{\Gamma(N/2+2)}\Big)G(x,x_0)\quad\mbox{in}\hspace{2mm}C^{1,\alpha}(\omega),
\end{equation}
where the convergence is in $C^{1,\alpha}(\omega)$ with $\omega$ any compact set of $\bar{\Omega}$ not containing $x_0$, $x_0$ is the limit point of $x_\epsilon$.
Owing to $\beta\neq0$, then similar to argument of \eqref{refor} in the proof of Theorem \ref{Figalli}.
\begin{equation*}
1-\lambda_{N+2,\varepsilon}=\frac{1}{\|u_{\epsilon}\|_{\infty}^{2}}\bigg[\frac{(N-2)R_N|\phi(x_0)|+\frac{2(p-q)}{p-1}B_{p,q}|\phi(x_0)|\mathcal{E}_N}{\mathcal{F}_N}+o(1)
\bigg]
\end{equation*}
where
$$\mathcal{C}_N=\frac{(N-2)R_N|\phi(x_0)|+\frac{2(p-q)}{p-1}B_{p,q}|\phi(x_0)|\mathcal{E}_N}{\mathcal{F}_N}<0.$$
Equality \eqref{rt} is thus established.
\end{proof}

\begin{proof}[Proof of Theorem \ref{thmprtb-1}]
We first prove that $\lambda_{N+2,\epsilon}$ is simple. Then we may assume by contradiction that there exist at least two eigenfunctions $v_{N+2,\epsilon}^{(1)}$ and $v_{N+2,\epsilon}^{(2)}$ corresponding to $\lambda_{N+2,\epsilon}$ orthogonal in the space $H_{0}^{1}(\Omega)$, and so that
\begin{equation*}
		\begin{split}	\int_{\Omega}\big(pu_{\epsilon}^{p-1}+\epsilon qu_{\epsilon}^{q-1}\big)v_{N+2,\epsilon}^{(1)}(x)v_{N+2,\epsilon}^{(2)}(x)dx=0.
\end{split}
\end{equation*}
Hence
\begin{equation*}
\begin{split}
\beta_1\beta_2\int_{\mathbb{R}^N}\Big(pW^{p-1}(x)+\epsilon qW^{q-1}(x) \Big)\Big(\frac{(N(N-2)-|x|^2}{(N(N-2)+|x|^2)^{\frac{N}{2}}}\Big)^2dx=0,
\end{split}
\end{equation*}
where $\beta_1$ and $\beta_2$ are the convergence coefficients of $v_{N+2,\epsilon}^{(1)}$ and $v_{N+2,\epsilon}^{(2)}$ as in \eqref{rtbdy46}, respectively. This gives a contradiction.

Next, we show that $v_{N+2,\epsilon}$ has only two nodal regions.  We can follow the same proof as in the last of \cite{GP05} except minor modifications.
Therefore, we will just sketch it. Define
\begin{equation}\label{gama}
h(x)=\frac{1-|x|^2}{(1+|x|^2)^{\frac{N}{2}}}>0\hspace{2mm}\mbox{in the ball}\hspace{2mm}B_1(0);\hspace{2mm}h(x)<0\hspace{2mm}\mbox{in}\hspace{2mm}\mathbb{R}^N\setminus B_1(0).
\end{equation}
Then we have that
$v_{N+2,\epsilon}>0\hspace{2mm}\mbox{in}\hspace{2mm}\big\{|x-x_{\epsilon}|<\frac{1}{2}\|u_{\epsilon}\|_{\infty}^{-\frac{p-1}{2}}\big\}.$
On the other hand, we find $v_{N+2,\epsilon}<0\hspace{2mm}\mbox{on} \hspace{2mm} \{|x-x_{\epsilon}|=\frac{1}{2}\|u_{\epsilon}\|_{\infty}^{-\frac{p-1}{2}}R\}.$
Thus, by the maximum principle and $v_{N+2,\epsilon}<0$ on $\Omega^{\prime\prime}$,
\begin{equation}\label{LL}
v_{N+2,\epsilon}<0\hspace{2mm}\mbox{in}\hspace{2mm}\Omega\setminus\{|x-x_{\epsilon}|<\frac{1}{2}\|u_{\epsilon}\|_{\infty}^{-\frac{p-1}{2}}R\}.
\end{equation}
Hence the conclusion follows by the argument of contradiction.

Finally, since \eqref{LL}, we find $v_{N+2,\epsilon}$ is always negative, and then the closure of the nodal set of $v_{N+2,\epsilon}$ does not touch the boundary.
This concludes the proof of Theorem \ref{thmprtb-1}.
\end{proof}	

\begin{proof}[Proof of Corollary \ref{emm-1}]
The conclusion follows by exploiting Theorem \ref{Figalli-1}, Theorem \ref{remainder terms} and \eqref{rt}.
\end{proof}
	
\section{The nondegeneracy result}
In this section we prove Theorem \ref{Figa}.
Let $\{\tilde{u}_{\epsilon}\}$ be a
family of solutions of \eqref{ele-1.1-1} which blow-up at a point $x_0$ and suppose $\tilde{u}_{\epsilon}$
assumes its maximum at $x_\epsilon$ such that $\lim\limits_{\epsilon\rightarrow0}x_{\epsilon}=x_0$.  Considering the rescaled function
\begin{equation}\label{1-miu}
\tilde{U}_{\epsilon}(x)=\frac{1}{\|\tilde{u}_{\epsilon}\|_{\infty}}\tilde{u}_{\epsilon}(\|u_{\epsilon}\|_{\infty}^{-\frac{p-1}{2}}x+x_{\epsilon})\quad\mbox{for}\quad x\in\tilde{\Omega}_{\epsilon}=\big\|\tilde{u}_{\epsilon}\big\|_{\infty}^{\frac{p-1}{2}}(\Omega-x_{\epsilon}).
\end{equation}	
In order to conclude the proof of Theorem \ref{Figa}, we argue by contradiction and let us assume that there exists $\tilde{w}_{\epsilon}$ $\not\equiv0$ which solves \eqref{nondegene} and by the $\tilde{W}_{\epsilon}$ the rescaled functions
\begin{equation}\label{1-1-miu}
\tilde{W}_{\epsilon}(x)=\frac{1}{\|\tilde{u}_{\epsilon}\|_{\infty}}\tilde{w}_{\epsilon}(\|\tilde{u}_{\epsilon}\|_{\infty}^{-\frac{p-1}{2}}x+x_{\epsilon})\quad\mbox{for}\quad x\in\tilde{\Omega}_{\epsilon}=\big\|\tilde{u}_{\epsilon}\big\|_{\infty}^{\frac{p-1}{2}}(\Omega-x_{\epsilon}).
\end{equation}	

Next, we give the rate of blow-up of the
maximum of the solutions, which will be used later to conclude the proof of Theorem \ref{Figa}. As aforementioned, Molle and Pistoia in \cite{Molle} studied the existence of solutions which blow-up and concentrate at a single point of $\Omega$. More precisely: there exist is a blow-up solution of \eqref{ele-1.1-1} having the form
$$
\tilde{u}_{\epsilon}=PW[x_{\epsilon},\mu_{\epsilon}^{-1}\epsilon^{\frac{2}{N-6+q(N-2)}}]+R_{\epsilon}
$$
which blows-up at a point $x_0$ with the rate of the concentration
$\mu_{0}$ such that $\mu_{\epsilon}^{-1}\rightarrow\mu_0^{-1}$, $x_{\epsilon}\rightarrow x_0\in\Omega$  and
$R_{\epsilon}\rightarrow0$ in $H_{0}^1(\Omega)$ as $\epsilon\rightarrow0$.
Then we find that $u_{\epsilon}$ are uniformly bounded near $\partial\Omega$ with respect to $\epsilon>0$ small and blow-up point an interior point of $\Omega$. Moreover, we have $\kappa$ is continuous on $\Omega$ and $\kappa(\|\tilde{u}_{\epsilon}\|_{\epsilon}^{-(p-1)/2}y+x_{\epsilon})\rightarrow\kappa(x_0)$ uniformly on compact sets of $\mathbb{R}^N$. Therefore, by using a similar proof as in Theorem 2 from \cite{HANZCHAO} (also see \cite{GM,Rey-1989}), we can establish the following asymptotic result and we omit the details.
\begin{thm}
Assume that $N\geq3$, $q\in(\max\{1,\frac{6-N}{N-2}\},2^{\ast}-1)$ and $\epsilon$ is sufficiently small. Let $\tilde{u}_\epsilon$ be a solution of \eqref{ele-1.1-1}.
 Then if $x_{\epsilon}$ is a point such that $\tilde{u}_{\epsilon}(x_{\epsilon})=\|\tilde{u}_{\epsilon}\|_{L^{\infty}(\Omega)}$, we have that $x_{\epsilon}\rightarrow x_0\in\Omega$ as $\epsilon\rightarrow0$ and the following results hold:
\begin{equation}\label{1-Fn}
\lim\limits_{\epsilon\rightarrow0}\epsilon\big\|\tilde{u}_\epsilon\big\|_{\infty}^{q+2-p}=B_{p,q}\frac{\phi(x_0)}{\kappa(x_0)},
\end{equation}
where $B_{p,q}$ can be found in \eqref{Fn}.
Moreover, it holds that
\begin{equation}\label{1-ifini}
\lim\limits_{\epsilon\rightarrow0^{+}}\big\|\tilde{u}_{\epsilon}\big\|_{\infty}\tilde{u}_{\epsilon}(x)=\frac{1}{N}(N(N-2))^{\frac{N}{2}}\omega_NG(x,x_{0})\hspace{2mm}\mbox{in} \hspace{2mm}C^{1,\alpha}(\Omega\setminus\{x_0\}).
\end{equation}
where $\omega_N$ is the area of the unit sphere in $\mathbb{R}^N$, and $G$ denotes the Green's function of the Laplacian $-\Delta$ in $\Omega$ with the Dirichlet boundary condition.
There exists a positive constant $C$, independent of $\epsilon$, such that
\begin{equation}\label{1-dus}
\tilde{u}_{\epsilon}(x)\leq C\big\|\tilde{u}_{\epsilon}\big\|_{\infty}\left(\frac{N(N-2)}{N(N-2)+\|\tilde{u}_{\epsilon}\|^{p-1}_{\infty}|x-x_{\epsilon}|^2}\right)^{\frac{N-2}{2}}.
\end{equation}
\end{thm}
\begin{lem}\label{1-UV1}
It holds
\begin{equation}\label{WSTAR}
|\tilde{W}_{\epsilon}(x)|\leq C\Big(\frac{1}{1+|y|^2/N(N-2)}\Big)^{\frac{N-2}{2}}\quad \mbox{for any}\quad x\in\mathbb{R}^N.
\end{equation}
\end{lem}
\begin{proof}
It is very similar the proof of \eqref{day2}, so we omit it here.
\end{proof}
Similar to the argument of Lemma \ref{UV1}, we have
\begin{lem}\label{UV1-1}
We have
$\tilde{U}_{\epsilon}\rightarrow W(x)$ and $\tilde{W}_{\epsilon}\rightarrow \tilde{w}_{0}$ in $C_{loc}^{2}(\mathbb{R}^N)$, where $\tilde{w}_0$ solves
\begin{equation}\label{WP}
-\Delta \tilde{w}_0=pW^{p-1}\tilde{w}_0\quad \mbox{in}\quad \mathbb{R}^N,
\end{equation}
and $$W(x)=\Big(\frac{1}{1+|x|^2/N(N-2)}\Big)^{\frac{N-2}{2}}.$$
\end{lem}
\begin{lem}\label{1-conseq}
It holds
\begin{equation*}
\begin{split}
&\int_{\partial\Omega}((x-y),\nu)\frac{\partial \tilde{u}_{\epsilon}}{\partial\nu}\frac{\partial \tilde{w}_{\epsilon}}{\partial\nu}dS_x=\frac{2(p-q)}{p-1}\epsilon \int_{\Omega}\kappa(x)\tilde{u}_{\epsilon}^q\tilde{w}_{\epsilon}dx+\epsilon\int_{\Omega}(x-y)\cdot\nabla\kappa(x)\tilde{u}_{\epsilon}^q\tilde{w}_{\epsilon}dx
\end{split}
\end{equation*}
for any $y\in\mathbb{R}^N$ and
where $\nu=\nu(x)$ denotes the unit outward normal to the boundary $\partial \Omega$.
\end{lem}
\begin{proof}
Set
\begin{equation}\label{wwxx}
\varpi(x):=\left((x-x_{\epsilon})\cdot\nabla \tilde{u}_{\epsilon}+\frac{2}{p-1}\tilde{u}_{\epsilon}\right).
\end{equation}
Then, for any $y\in\mathbb{R}^N$, we have
\begin{equation*}
\begin{split}
-\Delta \varpi(x)&=\frac{2p}{p-1}(-\Delta \tilde{u}_{\epsilon})-\sum_{j=1}^{N}\sum_{i\neq j}^{N}(x_i-y_j)\frac{\partial^3\tilde{u}_{\epsilon}}{\partial x_i\partial^2x_j}\\&
=\frac{2p}{p-1}\left(u_{\epsilon}^{p}+\epsilon\kappa(x) \tilde{u}_{\epsilon}^{q}\right)+\big(p\tilde{u}_{\epsilon}^{p-1}+\epsilon q\kappa(x)\tilde{u}_{\epsilon}^{q-1}\big)(x-y)\cdot\nabla \tilde{u}_{\epsilon}+\epsilon(x-y)\cdot\nabla\kappa(x)\tilde{u}_{\epsilon}^{q}\\&
=\big(p\tilde{u}_{\epsilon}^{p-1}+\epsilon q\kappa(x)\tilde{u}_{\epsilon}^{q-1}\big)\varpi(x)+\frac{2(p-q)}{p-1}\epsilon\kappa(x) \bar{u}_{\epsilon}^q+\epsilon(x-y)\cdot\nabla\kappa(x)\tilde{u}_{\epsilon}^q.
\end{split}
\end{equation*}
Multiplying this identity by $\tilde{w}_{\epsilon}$ and integrating by parts, we deduce
\begin{equation*}
\begin{split}
\int_{\Omega}\nabla \varpi(x)\cdot\nabla \tilde{w}_{\epsilon}dx=&\int_{\Omega}\big(p\tilde{u}_{\epsilon}^{p-1}+\epsilon q\kappa(x)\tilde{u}_{\epsilon}^{q-1}\big)\varpi(x)\tilde{w}_{\epsilon}dx\\&+\frac{2(p-q)}{p-1}\epsilon \int_{\Omega}\kappa(x)\tilde{u}_{\epsilon}^q\tilde{w}_{\epsilon}dx+\epsilon\int_{\Omega}(x-y)\cdot\nabla\kappa(x)\tilde{u}_{\epsilon}^q\tilde{w}_{\epsilon}dx.
\end{split}
\end{equation*}
On the other hand, we have
\begin{equation*}
			\begin{split}
\int_{\Omega}&\nabla \varpi(x)\cdot\nabla \tilde{w}_{\epsilon}dx-\int_{\partial\Omega}\Big((x-y)\cdot\nabla \tilde{u}_{\epsilon}\Big)\frac{\partial \tilde{w}_{\epsilon}}{\partial\nu}dS_x=\int_{\Omega}\big(p\tilde{u}_{\epsilon}^{p-1}+\epsilon q\kappa(x)\tilde{u}_{\epsilon}^{q-1}\big)\tilde{w}_{\epsilon}\varpi(x)dx.
\end{split}
		\end{equation*}
Combining these identities together yield the conclusion.
\end{proof}	
Hence similar to Lemma \ref{qiegao}, we also have
\begin{lem}\label{1-qiegao}
It holds
\begin{equation*}
\begin{split}
			\int_{\partial\Omega}\frac{\partial \tilde{u}_{\epsilon}}{\partial x_j}\frac{\partial \tilde{w}_{\epsilon}}{\partial\nu}dS_x=\epsilon\int_{\Omega} \tilde{u}_{\epsilon}^{q}\tilde{w}_{\epsilon}\frac{\partial \kappa(x)}{\partial x_j}dx\quad\mbox{for any}\quad j=1,\cdots,N.
\end{split}
\end{equation*}
	\end{lem}

We are now ready to prove Theorem \ref{Figa}.
\begin{proof}[Proof of Theorem \ref{Figa}]
Assume by contradiction that \eqref{nondegene} has a non-trivial solution $v_{\varepsilon}$, by Lemma \ref{UV1-1} and Proposition \ref{prondgr}, then there exists a $\tilde{w}_{0}$ such that $\tilde{W}_{\epsilon}\rightarrow \tilde{w}_{0}$ in $C_{loc}^{2}(\mathbb{R}^N)$ and $\tilde{w}_{0}$ solves \eqref{WP}. Moreover, since $\int_{\Omega_{\epsilon}}|\nabla \tilde{W}_{\epsilon}|\leq C$ by \eqref{WSTAR}, then by combining the Fatou's Lemma and Proposition \ref{prondgr}, we get that
\begin{equation}\label{w0}
\tilde{w}_{0}(x)=\sum_{k=1}^{N}\frac{\alpha_k x_k}{(N(N-2)+|x|^2)^{\frac{N}{2}}}+\beta\frac{N(N-2)-|x|^2}{(N(N-2)+|x|^2)^{\frac{N}{2}}}\quad \mbox{for some}\quad\alpha_k,~\beta\in\mathbb{R}.
\end{equation}
Next it is enough to prove $\beta=0$, $\alpha_k=0$ in \eqref{w0} and further deduce that $\tilde{w}_{0}=0$ cannot occur. It consists of three steps.

\begin{step}\label{Step 1.}
 $\beta=0$.
\end{step}
To this end, our purpose in what follows is to prove the following result.
\begin{Prop}\label{m00-1}
There exists $\beta\neq0$ such that
\begin{equation}\label{number-1}
\big\|\tilde{u}_{\epsilon}\big\|_{\infty}\tilde{w}_{\epsilon}(x)\rightarrow \beta p\frac{\omega_N}{2}\Big(\frac{\Gamma(\frac{N}{2})\Gamma(2)-\Gamma(\frac{N}{2}+1)\Gamma(1)}{\Gamma(N/2+2)}\Big)G(x,x_0)\quad\mbox{in}\quad C^{1,\alpha}(\omega),
\end{equation}
for $\alpha\in(0,1)$ as $\epsilon\rightarrow0$.
\end{Prop}
\begin{proof}
It is similar to proof of Lemma \ref{m00}, and so we omit it.
\end{proof}

By Lemma \ref{1-conseq}, we have
\begin{equation}\label{x+0}
\begin{split}
\int_{\partial\Omega}&((x-x_0),\nu)\frac{\partial (\|\tilde{u}_{\epsilon}\|_{\infty}\tilde{u}_{\epsilon})}{\partial\nu}\frac{\partial( \|\tilde{u}_{\epsilon}\|_{\infty}\tilde{w}_{\epsilon})}{\partial\nu}dS_x\\&=\frac{2(p-q)}{p-1}\epsilon \|\tilde{u}_{\epsilon}\|_{\infty}^2\int_{\Omega}\kappa(x)\tilde{u}_{\epsilon}^q\tilde{w}_{\epsilon}dx+\epsilon\|\tilde{u}_{\epsilon}\|_{\infty}^2\int_{\Omega}(x-y)\cdot\nabla\kappa(x)\tilde{u}_{\epsilon}^q\tilde{w}_{\epsilon}dx.
\end{split}
\end{equation}
From \eqref{GX0}, \eqref{1-ifini} and \eqref{number-1}, we obtain
\begin{equation}\label{x+0-1}
LHS\hspace{2mm}\mbox{of}\hspace{2mm}\eqref{x+0}\rightarrow\beta(N-2)R_N|\phi(x_0)|,
\end{equation}
where $R_N$ can be found in the beginning of proof of Theorem \ref{Figalli}.
On the other hand, we first denote $$\mathcal{K}(x):=\frac{2(p-q)}{p-1}\kappa(x)+(x-y)\cdot\nabla\kappa(x).$$
Then we have $\mathcal{K}$ is continuous on $\Omega$ and $\mathcal{K}(\|u_{\epsilon}\|_{\epsilon}^{-\frac{p-1}{2}}y+x_{\epsilon})\rightarrow\frac{2(p-q)}{p-1}\kappa(x_0)$ uniformly on compact sets of $\mathbb{R}^N$. Thus, by the oddness of the integrand, \eqref{1-Fn} and \eqref{w0}, we obtain
\begin{equation}\label{x+0-2}
\begin{split}
&RHS\hspace{2mm}\mbox{of}\hspace{2mm}\eqref{x+0}\\&\rightarrow\epsilon\|\tilde{u}_{\epsilon}\|_{\infty}^{q+2-p}
\frac{2(p-q)}{p-1}\kappa(x_0)\int_{\mathbb{R}^N}W^q\Big[\sum_{k=1}^{N}\frac{\alpha_k y_k}{(N(N-2)+|y|^2)^{\frac{N}{2}}}+\beta\frac{N(N-2)-|y|^2}{(N(N-2)+|y|^2)^{\frac{N}{2}}}\Big]dy\\&
=\beta\frac{2(p-q)}{p-1}B_{p,q}\phi(x_0)\int_{\mathbb{R}^N}W^q\frac{N(N-2)-|y|^2}{(N(N-2)+|y|^2)^{\frac{N}{2}}}dy+o(1)\\&
=\beta\frac{2(p-q)}{p-1}B_{p,q}\mathcal{E}_N\phi(x_0),
\end{split}
\end{equation}
where $R_N$ can be found in the proof of Theorem \ref{Figalli}.
Coupling \eqref{x+0-1} and \eqref{x+0-2}, we conclude that, if $\beta\neq0$,
$$-(N-2)R_N=\frac{2(p-q)}{p-1}B_{p,q}\mathcal{E}_N.$$
Recalling the argument of four cases on $q\in(\max\{1,\frac{6-N}{N-2}\},p)$ if $N\geq3$, we have $\frac{2(p-q)}{p-1}B_{p,q}\mathcal{E}_N+(N-2)R_N<0$,
a contradiction. Hence we complete the proof of first step.
\begin{step}
$\alpha_k=0$.
\end{step}
 At this point it essentially remains to prove the following result.
\begin{Prop}\label{6.3lemma-1}
Assume that $1<q<p$ if $N\geq5$, $\frac{3}{2}<q<p$ if $N=4$. Let $\beta=0$ and $\alpha_k\neq0$ in \eqref{w0}. Then
\begin{equation*}
\big\|\tilde{u}_{\epsilon}\big\|_{\infty}^{\frac{p+1}{2}}\tilde{w}_{\epsilon}=\frac{1}{N-2}\int_{\mathbb{R}^N}W^pdy\sum_{k=1}^{N}\alpha_{k}\left(\frac{\partial}{\partial y_k}G\left(x,y\right)\right)\Big|_{y=x_{0}}
+o(1)\quad\mbox{in}\quad C_{loc}^1\big(\overline{\Omega}\setminus\{x_0\}\big).
\end{equation*}
\end{Prop}
\begin{proof}
The proof of this Lemma is very similar to that of Lemma \ref{6.3lemma}. Therefore, we will just sketch it.
By Green's representation formula, we get
\begin{equation}\label{115if}
\begin{split}
\tilde{w}_{\epsilon}(x)&=p\int_{\Omega}G(x,y)u_{\epsilon}^{p-1}\tilde{w}_{\epsilon}dx+\epsilon q\int_{\Omega}\kappa(x)u_{\epsilon}^{q-1}\tilde{w}_{\epsilon}dx
=:J_1+J_2.
\end{split}
\end{equation}
Next we consider two cases, depending choice of $N$ and $q$.

\textbf{Case 1:} $\frac{N}{N-2}<q<p$ if $N\geq4$.

For $J_1$. Similar to the case 1 of $I_1$ in Lemma \ref{6.3lemma} except some minor modifications, we have
\begin{equation*}
\begin{split}
J_1+J_2
=\frac{1}{N-2}\frac{1}{\|u_{\epsilon}\|_{\infty}^{(p+1)/2}}\int_{\mathbb{R}^N}W^pdy\sum_{k=1}^{N}\alpha_{k}^{i}\left(\frac{G\left(x,y\right)}{\partial y_k}\right)\big|_{y=x_0}dy+o(1)\hspace{2mm}\mbox{in}\hspace{2mm}\overline{\Omega}\setminus\{x_0\}.
\end{split}
\end{equation*}
by apply the map $y=\mathcal{B}\xi\Longleftrightarrow y_k=\sum_{l=1}^{N}\beta_{kl}\xi_l$, where $\mathcal{B}=(\beta_{kl})_{k,l=1,\cdots,N}$ is an orthogonal matrix which maps the vector $(|\alpha|,0,\cdots,0)$ into $\alpha=(\alpha_1,\cdots,\alpha_N)$.

\textbf{Case 2:} $\min\{\frac{3}{N-2},1\}<q\leq\frac{N}{N-2}$ if $N\geq4$.

Similarly to case 2 of $I_1$ in Lemma \ref{6.3lemma}, we also have
$$
J_1=\frac{1}{N-2}\frac{1}{\|u_{\epsilon}\|_{\infty}^{(p+1)/2}}\int_{\mathbb{R}^N}W^pdy\sum_{k=1}^{N}\alpha_{k}^{i}\left(\frac{G\left(x,y\right)}{\partial y_k}\right)\big|_{y=x_0}dy+o(1)\hspace{2mm}\mbox{in}\hspace{2mm}\overline{\Omega}\setminus\{x_0\}.
$$
by some computations.
For $J_2$, we find
\begin{equation*}
\begin{split}
J_2=\frac{\epsilon q}{\|u_{\epsilon}\|_{\infty}^{p+1-q}}
\int_{\Omega_\epsilon}\tilde{U}_{\epsilon}^{q-1}(y)\tilde{W}_{\epsilon}(y) G\left(x,\frac{y}{\big\|u_{\epsilon}\big\|_{\infty}^{\frac{p-1}{2}}}+x_{\epsilon}\right)dy. \end{split}
\end{equation*}
By Lemma \ref{Ve}, there exists a unique solution $\bar{w}$ of the following initial value problem and combining $\Omega_{\epsilon}\subset B_{c\|u_{\epsilon}\|_{\infty}^{(p-1)/2}}(0)$ for some $c>0$, we get
\begin{equation*}
\begin{split}
\Big|\int_{\Omega_{\epsilon}}\bar{w}_{\epsilon}(y)dy\Big|
\leq C\big\|u_{\epsilon}\big\|_{\infty}^{\frac{p-1}{2}(N+1-(N-2)q)},
\end{split}
\end{equation*}
where $1<q\leq\frac{N}{N-2}$. Hence
\begin{equation*}
\begin{split}
\big\|u_{\epsilon}\big\|_{\infty}^{\frac{p+1}{2}}|I_2|&=\frac{\lambda_{i,\epsilon}\epsilon q}{\|u_{\epsilon}\|_{\infty}^{p+1-q-(p+1)/2}}
\int_{\Omega_\epsilon}\sum_{k=1}^{N}\alpha_{k}\frac{\bar{w}(y)}{\partial y_k} G\left(x,\frac{y}{\big\|u_{\epsilon}\big\|_{\infty}^{\frac{p-1}{2}}}+x_{\epsilon}\right)dy\\&
\leq C\frac{1}{\|u_{\epsilon}\|_{\infty}^{2}}\big|\int_{\Omega_\epsilon} \bar{w}_{\epsilon}(y)dy\big|\leq C\frac{1}{\|u_{\epsilon}\|_{\infty}^{2(q-\frac{3}{N-2})}}=o(1),
\end{split}
\end{equation*}
where $q>\frac{3}{N-2}$.
Combining the estimates of $J_1$ and $J_2$ in the above two cases, we conclude the proof.
\end{proof}

We now turn to the proof of second step.
By the previous step we get
\begin{equation*}
\tilde{W}_{\epsilon}(y)\rightarrow\sum_{k=1}^{N}\frac{\alpha_k y_k}{(N(N-2)+|y|^2)^{\frac{N}{2}}}\quad\mbox{in}\quad C_{loc}^{1}(\mathbb{R}^N).
\end{equation*}
Then we have
$$\tilde{U}_{\epsilon}^{q}(y)\tilde{W}_{\epsilon}(y)\rightarrow-\frac{1}{(q+1)(N-2)}\sum_{k=1}^{N}\alpha_k \frac{\partial}{ \partial y_k}W^{q+1}(y).$$
Thus Lemma \ref{Ve} tell us that there exist a solution $\bar{w}_{\epsilon}$ to problem \eqref{kapax} for $f:=\tilde{U}_{\epsilon}^{q}(y)\tilde{W}_{\epsilon}(y)$ such that
\begin{equation}\label{wyifa}
\int_{\mathbb{R}^N}\bar{w}_{\epsilon}(y)dy\rightarrow-\int_{\mathbb{R}^N}\frac{W^{q+1}(y)}{(q+1)(N-2)}dy.
\end{equation}
Next similar to the proof of Theorem \ref{remainder terms}, and  by the virtue of Lemma \ref{1-qiegao}, Proposition \ref{6.3lemma-1}, \eqref{1-ifini} and \eqref{wyifa} we have
\begin{equation*}
\begin{split}
&\frac{(N(N-2))^{\frac{N-2}{2}}}{2}\omega_N\int_{\mathbb{R}^N}W^pdx\sum_{k=1}^{N}\alpha_{k}\Big(\frac{\partial^2\phi}{\partial x_k\partial x_j}(x_0)\Big)\\&=
			\int_{\partial\Omega}\frac{\partial (\|\tilde{u}_{\epsilon}\|_{\infty}\tilde{u}_{\epsilon})}{\partial x_j}\frac{\partial (\|\tilde{u}_{\epsilon}\|_{\infty}^{\frac{p+1}{2}}\tilde{w}_{\epsilon})}{\partial\nu}dS_x
=\epsilon\big\|\tilde{u}_{\epsilon}\big\|_{\infty}^{\frac{p+3}{2}}\big\|\tilde{u}_{\epsilon}\big\|_{\infty}^{q-p}\int_{\Omega_{\epsilon}} \tilde{U}_{\epsilon}^{q}\tilde{W}_{\epsilon}\frac{\partial \kappa}{\partial x_j}\big(\big\|\tilde{u}_{\epsilon}\big\|_{\infty}^{-\frac{p-1}{2}}y+x_{\epsilon}\big)dy\\&
=-\epsilon\big\|\tilde{u}_{\epsilon}\big\|_{\infty}^{q+\frac{3-p}{2}}\big\|\tilde{u}_{\epsilon}\big\|_{\infty}^{-\frac{p-1}{2}}\int_{\Omega_{\epsilon}}
\bar{w}_{\epsilon}(y)\sum_{k=1}^{N}\alpha_k \frac{\partial^2 \kappa}{\partial x_k\partial x_j}\big(\big\|\tilde{u}_{\epsilon}\big\|_{\infty}^{-\frac{p-1}{2}}y+x_{\epsilon}\big)dy\\&
=\frac{B_{p,q}}{(q+1)(N-2)}\int_{\mathbb{R}^N}W^{q+1}(y)dy\frac{\phi(x_0)}{\kappa(x_0)}
\sum_{k=1}^{N}\alpha_k\Big(\frac{\partial^2 \kappa}{\partial x_k\partial x_j}(x_0)\Big),
\end{split}
\end{equation*}
where the last equality we used the limitation \eqref{1-Fn}. This implies that
$$
\sum_{k=1}^{N}\alpha_{k}\bigg[\frac{1}{\kappa(x_0)}\Big(\frac{\partial^2 \kappa}{\partial x_k\partial x_j}(x_0)\Big)-\Gamma_{p,q}\frac{1}{\phi(x_0)}\Big(\frac{\partial^2\phi}{\partial x_k\partial x_j}(x_0)\Big)\bigg]=0,
$$
where
$$\Gamma_{p,q}=\frac{a_N(p-q+1)N}{2b_N}\frac{\Gamma(\frac{N}{2})\Gamma(\frac{(N-2)(q+1)-N}{2})}{\Gamma(\frac{(N-2)(q+1)}{2})}$$
with
$$a_N=\int_{0}^{\infty}\frac{r^{N-1}dr}{(N(N-2)+r^2)^{\frac{N+2}{2}}}
\quad\mbox{and}\quad b_N=\int_{0}^{\infty}\frac{r^{N-1}dr}{(N(N-2)+r^2)^{\frac{N-2}{2}(q+1)}}.$$
Since matrix \eqref{CC} at $x_0$ is non-degenerate, we get that $\alpha_k=0$.

\begin{step}
$\tilde{w}_{0}=0$ cannot occur.
\end{step}

By the previous steps we have that $\tilde{W}_{\epsilon}\rightarrow\tilde{w}_{0}\equiv0$ as $\epsilon\rightarrow0$. Let $\xi_\epsilon$ be a the point where $\tilde{W}_{\epsilon}$ achieves its maximum, i.e. $\tilde{W}_{\epsilon}(\xi_{\epsilon})=\|\tilde{W}_{\epsilon}\|_{\infty}=1$.
Then necessarily $|\xi_{\epsilon}|\rightarrow\infty$, this is prevented by
Lemma \ref{1-UV1}. Theorem \ref{Figa} is proven.
\end{proof}


\small

\begin{thebibliography}{99}

\bibitem{ATKINSON-1986}
F. Atkinson and L. Peletier,
 \emph{Emder-Fowler equations involving critical exponents},
Nonlinear Anal. TMA., {\bf 10} (1986), 755--776.

\bibitem{Bahri-1988}
A. Bahri and J. Coron,
 \emph{On a nonlinear elliptic equation involving the critical {S}obolev exponent: the effect of the topology of the domain}, Comm. Pure Appl. Math.,
{\bf 41}(1988), 253--294.

\bibitem{BE91}
G. Bianchi, H. Egnell, {\em A note on the Sobolev inequality},  J. Funct. Anal., {\bf 100}(1), 18-24, 1991.

\bibitem{Brezis-1983}
H. Brezis and L. Nirenberg,
 \emph{Positive solutions of nonlinear elliptic equations involving critical Sobolev exponents}, Comm. Pure Appl. Math.,
{\bf 36}(1983), 437--477.

\bibitem{BP} H. Brezis and L. A. Peletier,
         \emph{Asymptotics for elliptic equation involving critical growth, in: Partial Differential Equations and the Calculus of Variations},
vol. I, Birkh\"{a}user, Boston, MA, 1989, PP. 149--192.


\bibitem{B-L-R}
A. Bahri, Y. Y. Li, and O. Rey,
\emph{On a variational problem with lack of compactness: the topological effect of the critical points at infinity}, Calc. Var. Partial Differ. Equ.,
{\bf 3} (1995), 67--93.

\bibitem{CKIM-1}
W. Choi and S. Kim,
\emph{Asymptotic behavior of least energy solutions to the Lane-Emden system near the critical hyperbola}, J. Math. Pures Appl. {\bf 132} (2019), 398--456.

\bibitem{CKL}
W. Choi, S. Kim, and K.-A. Lee,
\emph{Asymptotic behavior of solutions for nonlinear
elliptic problems with the fractional Laplacian}, J. Funct. Anal., {\bf 266}, 6531--6598 (2014).

\bibitem{CKL-1}
W. Choi, S. Kim, and K.-A. Lee,
\emph{Qualitative properties of multi-bubble solutions for nonlinear elliptic equations involving critical exponents}, Adv. Math., {\bf 298}(2016), 484--533.


\bibitem{dgp}
L. Damascelli, M. Grossi, and F. Pacella, \emph{Qualitative properties of positive solutions of semilinear elliptic equations in symmetric
domains via the maximum principle}, Ann. Inst. Henri Poincar\'{e}, Anal. Non Lin\'{e}aire {\bf 16} (5) (1999) 631--652.


\bibitem{Ding1989}
W. Ding,
 \emph{Positive solutions of $\Delta u+u^{\frac{n+2}{n-2}}=0$ on contractible domains}, J. Differential Equations, {\bf 2}(1989), 83--88.

\bibitem{GM}
J. Gao and S. Ma, \emph{Asymptotic behavior for the Brezis-Nirenberg problem. The subcritical perturbation case}, Preprint. \url{arXiv:2502.16505} [math.AP].

\bibitem{GNN}
B. Gidas, W.M. Ni, and L. Nirenberg,
 \emph{Symmetry and related properties via the maximum principle}, Comm. Math. Phys. {\bf68}(1979), 209--243.

\bibitem{GG09}
F. Gladiali and M. Grossi, \emph{On the spectrum of a nonlinear planar problem}, Ann. Inst. H. Poincar\'{e} Anal. Non Lin\'{e}aire {\bf 26} (2009), 191--222.

\bibitem{GGO}
F. Gladiali, M. Grossi, and H. Ohtsuka,
\emph{On the number of peaks of the eigenfunctions of the linearized
Gel'fand problem}, Ann. Mat. Pura Appl., {\bf 195} (2016), 79--93.

\bibitem{GGO-1}
F. Gladiali, M. Grossi, H. Ohtsuka, and T. Suzuki,
\emph{Morse indices of multiple blow-up solutions to the
two-dimensional Gel'fand problem}, Comm. Partial Differential Equations,
{\bf 39} (2014), 2028--2063.

\bibitem{GP05}
M. Grossi and F. Pacella, \emph{On an eigenvalue problem related to the critical exponent}, Math. Z., {\bf 250} (2005) 225--256.


\bibitem{HANZCHAO} Z. Han,
\emph{Asymptotic approach to singular solutions for nonlinear elliptic equations involving critical Sobolev exponent}, Ann. Inst. Henri Poincar\'{e}, Anal. Non Lin\'{e}aire {\bf 8} (1991), 159--174.


\bibitem{LiWZ}
H. Li, J. Wei, and W. Zou, {\em Uniqueness, multiplicity and nondegeneracy of positive solutions to the Lane-Emden problem}, J. Math. Pures Appl. {\bf 179} (2023) 1--67.

\bibitem{LTX}
P, Luo, Z, Tang, and H, Xie,
\emph{Qualitative analysis to an eigenvalue problem of the Hénon equation},
 J. Funct. Anal., {\bf 286} (2024), no. 2, Paper No. 110206, 26 pp.

\bibitem{Merle}
F. Merle and L. A. Peletier, \emph{Asymptotic behaviour of positive solutions of elliptic equations with critical and supercritical growth}, II. The nonradial case, J. Funct. Anal., {\bf 105} (1992), 1--41.

\bibitem{Molle}
R. Molle and A. Pistoia, \emph{Concentration phenomena in elliptic problem with critical and supercritical growth}, Adv. Differential Equations, {\bf 8}(2003), 547--570.

\bibitem{Musso-Pistoia-2002}
 M. Musso and A. Pistoia,
\emph{Multispike solutions for a nonlinear elliptic problem involving the critical Sobolev exponent}, Indiana Univ. Math. J., {\bf 51} (2002), 541--579.

\bibitem{Musso-Pistoia-2003}
M. Musso and A. Pistoia,
\emph{Double blow-up solutions for a Brezis-Nirenberg type problem}, Commun. Contemp. Math.,  {\bf 5} (2003), no. 5, 775--802.

\bibitem{Pohozaev-1965}
S. Poho\v{z}aev,
\emph{Eigenfunctions of the equation $\Delta u=\lambda f(u)$},
Soviet Math. Dokl., \textbf{6}(1965), 1408--1411.

\bibitem{Rey-1989}
O. Rey, \emph{Proof of two conjectures of H. Brezis and L. A. Peletier},
Manus. Math., \textbf{65} (1989), 19--37.

\bibitem{Rey-1990}
O. Rey, \emph{The role of the Green's function in a nonlinear elliptic equation involving the critical Sobolevexponent}, J. Funct. Anal., \textbf{89} (1990),1--52.


\bibitem{T}
F. Takahashi, \emph{Asymptotic nondegeneracy of the least energy solutions to an elliptic problem with the critical Sobolev exponent coefficients}, Adv. Nonlinear Stud., \textbf{8} (2008), 783--798.

\bibitem{T-1}
F. Takahashi,
\emph{An eigenvalue problem related to blowing-up solutions for a semilinear elliptic equation with the critical Sobolev exponent}, Discrete Contin. Dyn. Syst. Ser. S, \textbf{4} (2011), 907-922.

\bibitem{HL}
X. Wang, \emph{On location of blow-up of ground states of semilinear elliptic equations in $R^{N}$ involving critical sobolev exponent}, J. Differential Equations, \textbf{127}, 148--173 (1996).

\bibitem{JW0}
J, Wei,
\emph{Asymptotic behavior of least energy solutions to a semilinear Dirichlet problem near the critical exponent},
J. Math. Soc. Japan, Vol. 50, No. 1, 1998.

\bibitem{WY}
J. Wei and S. Yan,
\emph{Infinitely many solutions for the prescribed scalar curvature problem on $S^N$}, J. Funct. Anal., \textbf{258}, 3048--3081 (2010).

\bibitem{Widman}
 K.O. Widman, \emph{Inequalities for the Green function and boundary continuity of
the gradient of solutions of elliptic differential equations}, Math. Scand. \textbf{21}, 17--37 (1967).

\end{thebibliography}
\end{document}